\documentclass[12pt,a4paper]{amsart}

\usepackage{amsmath,amsthm,amssymb,latexsym,a4wide,tikz,multicol}
\usepackage{arydshln,multirow}
\usepackage{tikz-qtree,tikz-qtree-compat}
\usepackage{tikz}
\tikzset{font=\small}

\newtheorem{theorem}{Theorem} 
\newtheorem{lemma}[theorem]{Lemma}
\newtheorem{corollary}[theorem]{Corollary}
\newtheorem{proposition}[theorem]{Proposition}

\newtheorem{example}[theorem]{Example}
\newtheorem{remark}[theorem]{Remark}
\usepackage{tikz}
\usepackage{enumerate}
\usetikzlibrary{matrix,arrows}
\usetikzlibrary{positioning}
\usetikzlibrary{patterns}

\usepackage[active]{srcltx}

\usepackage{enumerate}
\numberwithin{equation}{section}

\DeclareRobustCommand{\stirling}{\genfrac\{\}{0pt}{}}

\title{Free skew Boolean intersection algebras and set partitions}

\author{Ganna Kudryavtseva}
\address{G. Kudryavtseva: University of Ljubljana,
Faculty of Civil and Geodetic Engineering, Jamova cesta~2, SI-1000 Ljubljana, Slovenia / Institute of Mathematics, Physics and Mechanics, Jadranska ulica 19, SI-1000 Ljubljana, Slovenia /Jo\v zef  Stefan Institute,
Jamova cesta 39, SI-1000 Ljubljana, Slovenia}
\email{ganna.kudryavtseva\symbol{64}fgg.uni-lj.si, ganna.kudryavtseva\symbol{64}imfm.si, ganna.kudryavtseva\symbol{64}ijs.si}

\thanks{The author was partially supported by  ARRS grant P1-0288.}

\begin{document}

\begin{abstract} We show that atoms of the $n$-generated free left-handed skew Boolean intersection algebra are in a bijective correspondence with pointed partitions of non-empty subsets of $\{1,2,\dots, n\}$.   Furthermore, under the canonical inclusion into the $k$-generated free algebra, where $k\geq n$, an atom of the $n$-generated free algebra decomposes into an orthogonal join of atoms of the $k$-generated free algebra in an agreement with the containment relation on the respective partitions.  As a consequence of these results, we describe the structure of finite free left-handed skew Boolean intersection algebras and express several their combinatorial characteristics in terms of Bell numbers and Stirling numbers of the second kind.  We also look at the infinite case. For countably many generators, our constructions lead to the `partition analogue' of the Cantor tree whose boundary is the `partition variant' of the Cantor set. \end{abstract}

\maketitle

\section{Introduction}\label{s0:introduction}

{\em Skew Boolean intersection algebras} (SBIAs) are non-commutative variants of generalized Boolean algebras (GBAs). These are algebras $(S; \wedge,\vee,\setminus, \sqcap, 0)$ of signature $(2,2,2,2,0)$. The operations $\wedge$, $\vee$ and $\setminus$ are  variants of the meet, join and difference operations in a generalized Boolean algebra, respectively. In general, both $\wedge$ and $\vee$ are non-commutative. The operation $\sqcap$ is the commutative {\em intersection} operation, which is another generalization of the meet operation tied to the natural partial order underlying the algebra. When an SBIA is commutative (meaning that both $\wedge$ and $\vee$ are commutative), the operations $\wedge$ and $\sqcap$ coincide and each becomes the usual meet operation.

Skew Boolean intersection algebras were first extensively studied by Bignall and Leech in \cite{BL}, where their close connection with discriminator varieties of universal algebra was established. Removing the intersection operation from the signature of an SBIA turns it into a {\em skew Boolean algebra} called the {\em SBA reduct} of the initial SBIA. Skew Boolean algebras (SBAs) are another non-commutative variants of generalized Boolean algebras (see Leech \cite{L2,L3}). Each such an algebra possesses a natural underlying partial order. If binary meets exist for this partial order, such a meet of $a$ and $b$ is called the {\em intersection} of $a$ and $b$, and the SBA is said to have {\em  intersections}. Upon adding the intersection operation to the signature of the SBA, it is turned into an SBIA.

Skew Boolean intersection algebras are natural and frequently appearing objects. Any finite SBA (or, more generally, any SBA whose maximal commutative quotient is finite) possesses intersections and can thus be looked at as an SBIA. Moreover, under the non-commutative Stone duality \cite{Kud} between left-handed skew Boolean algebras and \'etale spaces over locally compact Boolean spaces, SBAs possessing intersections can be characterized precisely as those for which the dual \'etale space is Hausdorff (see also \cite{BCV, KudLaw}).  The relationship of skew Boolean intersection algebras and another class of algebras axiomatizing override and update operations on functions was studied by Cvenko-Vah, Leech and Spinks in \cite{CVLS}. Cvetko-Vah and Salibra \cite{CVS} revealed  the connection between skew Boolean intersection algebras and Church algebras.  Varieties of skew Boolean intersection algebras have been  studied by Leech and Spinks in \cite{LS}.

Free skew Boolean algebras have been recently systematically studied by Leech and the author in \cite{KL}. There it was shown that a free SBA over $X$ can be looked at as a certain `upgrade' of the free GBA over $X$.
In particular, in the case where the generating set $X$ is finite of cardinality $n$, the atoms of the free left-handed SBA over $X$ are in a bijective correspondence with {\em pointed} non-empty subsets of $X$, that is, pairs $(A,a)$ where $A$ is a non-empty subset of $X$ and $a\in A$. It is shown in \cite{KL}  that for any generating set $X$, the free SBA over $X$ has intersections, however it is not free as an SBIA, as the intersection of any two generators equals $0$. This raises the question to study the structure of free SBIAs.

In the present paper we  demonstrate that the structure of free left-handed SBIAs upgrades that of the `partition analogues' of free GBAs. 
In particular, atoms of the free left-handed SBIA over a finite set $X$ are in a bijective correspondence with {\em pointed partitions} of non-empty subsets of $X$. On the set of all pairs $(Y,\alpha)$ where $Y$ is a non-empty subset of $X$ and $\alpha$ is a partition of a non-empty subset of $Y$, there is a natural {\em containment order} given by  $(Y,\alpha)\preceq (Z,\beta)$  if ${\mathrm{dom}}(\alpha)\subseteq {\mathrm{dom}}(\beta)$, $Y\setminus {\mathrm{dom}}(\alpha)\subseteq Z\setminus {\mathrm{dom}}(\beta)$, and $x\mathrel{\alpha} y$ if and only if $x\mathrel{\beta} y$ for any $x,y\in {\mathrm{dom}}(\alpha)$. We show that, under the canonical inclusion into the free left-handed SBIA over a finite set $Y$ where $X\subseteq Y$, the rule for decomposing of an atom over $X$ into an orthogonal join of atoms over $Y$ is governed by the  containment order on the corresponding partitions.

The structure of the paper is as follows. In Section \ref{s:prelim} we collect the background material on skew Boolean intersection algebras and explain that results for general SBIAs can be easily derived from  respective result  for the left-handed case which is the most important and to which the considerations in Sections \ref{s:normal_forms} and \ref{s:free}  are restricted. In Section~\ref{s:normal_forms} we give a construction of elements $e(X,\alpha, A)$, determined by pointed partitions $(X,\alpha, A)$, the main construction of the paper. These elements can be looked at as subtle generalizations of elementary conjunctions in a Boolean algebra. We call them {\em elementary elements}. We then study their properties and in Theorem~\ref{th:branching} we prove the Decomposition Rule, our crucial result, which was outlined in the previous paragraph. We do not require that the algebra is free to prove this result, just in a non-free algebra $(X,\alpha, A)$ can not be always reconstructed from  $e(X,\alpha, A)$ (in fact, this happens precisely if $e(X,\alpha, A)=0$). This leads to the theory of normal forms, which is summarized in Theorem \ref{th:normal_forms}. Section \ref{s:normal_forms} is concluded by pointing out that in the commutative case, that is, the case where $\sqcap=\wedge$, our theory reduces to the usual theory of disjunctive normal forms in generalized Boolean algebras. In Section~\ref{s:free} we turn to free left-handed SBIAs. In Proposition \ref{prop:free} we observe that, for $X$ finite, $S=\langle X\rangle$ is free over $X$ if and only if all elementary elements over $X$ are non-zero and pairwise distinct.  This leads to Theorem~\ref{th:combinatorial} where we describe the structure of finite free left-handed SBIAs and calculate several their combinatorial characteristics, such as the number of atomic ${\mathcal D}$-classes, the number of atoms and  cardinalities of  all the ${\mathcal D}$-classes. As soon as set partitions are involved, the characteristics are expressed in terms of Bell numbers and Stirling numbers of the second kind. We also show that the center of the free $n$-generated left-handed SBIA (and thus also its maximal commutative quotient) is isomorphic to the free $n$-generated GBA. Theorem~\ref{th:5.8} provides an explicit construction of infinite free algebras. It is followed by some of their properties. In the case, where the generating set $X$ is countable, the GBA underlying the SBA reduct of the free left-handed SBIA over $X$, is the dual GBA of the boundary (with one point removed) of the infinite partition tree, the `partition analogue' of the Cantor tree. Since the partition tree is Cantorian, its boundary is homeomorphic to the Cantor set. Thus the partition analogue of the free GBA over $X$ is isomorphic to the free GBA on $X$ itself. This contrasts the fact (see Corollary \ref{cor:free11}) that, for a finite generating set, the partition analogue of the free GBA is not free.

\section{Preliminaries}\label{s:prelim}

In this section we provide the  background material on skew Boolean intersection algebras and set up some notation. For a more extensive exposition of the background material and more results on skew lattices we refer the reader to \cite{BL, L2,L6,L3}.

\subsection{Skew Boolean algebras}\label{sub:skew}

A {\em  skew Boolean algebra} (or an {\em SBA}) is an algebra $(S;\wedge,\vee,\setminus, 0)$ of type $(2,2,2,0)$ such that the following identities hold:
\begin{enumerate}
\item \label{axs:associat} (associativity) $x\vee (y\vee z) = (x\vee y)\vee z$, $x\wedge (y\wedge z) = (x\wedge y)\wedge z$;
\item \label{axs:abs} (absorption) $x\vee (x\wedge y)=x=(y\wedge x)\vee x$ and $x\wedge (x\vee y)= x = (y\vee x)\wedge x$;
\item \label{axs:dist} (distributivity) $x\wedge (y\vee z)= (x\wedge y)\vee (x\wedge z)$, $(y\vee z)\wedge x=(y\wedge x)\vee (z\wedge x)$;
\item \label{axs:zero} (properties of $0$) $x\vee 0=0\vee x = x$, $x\wedge 0 = 0\wedge x =0$;
\item \label{axs:comp} (properties of relative complement) $(x\setminus y)\wedge y=y\wedge (x\setminus y)=0$, \\ $(x\setminus y)\vee (x\wedge y\wedge x)=x=(x\wedge y\wedge x)\vee (x\setminus y)$;
\item \label{axs:normality} (normality) $x\wedge y\wedge z\wedge t = x\wedge z\wedge y\wedge t$.
\end{enumerate}

Associativity and absorption tell us that $(S;\wedge,\vee)$ is a {\em skew lattice}, a non-commutative variant of a lattice.
Normality axiom is a weakened form of commutativiity. The absorption axiom, just as for lattices, implies:

\begin{enumerate}
\item[(7)] (idempotency) $x\vee x=x=x\wedge x$.
\end{enumerate}

Thus an SBA is a relatively complemented normal and distributive skew lattice with a zero element. GBAs may be characterised precisely as commutative SBAs.
Since SBAs are defined by identities, they form a variety of algebras.

To simplify notation, we take a convention to refer to a skew lattice $(S;\wedge,\vee)$ or a skew Boolean algebra $(S;\wedge,\vee,\setminus, 0)$ etc. just by $S$, thus making no notational difference between an algebra and its underlying set.

 The {\em underlying partial order} $\leq$ on a skew lattice $S$ is defined by $x\leq y$ if and only if $x\wedge y=y\wedge x=x$ or, equivalently, $x\vee y=y\vee x=y$. It plays a major role in this paper, as the intersection operation $\sqcap$, introduced later, coincides with the meet operation with respect to this order.

Normality implies that  the {\em principal subalgebra} $a^{\downarrow}=\{b\in S\colon b\leq a\}$
is a lattice for any element $a$ of an SBA $S$. Due to the presence of the zero element,  relative complementation and distributivity, each  $a^{\downarrow}$ is a GBA with the top element $a$.

A skew lattice $S$ is {\em rectangular} if it satisfies the identity $x\wedge y\wedge x=x$ or, equivalently, its dual
 $x\vee y\vee x=x$. 
 Given nonempty sets  $L$ and $R$, a rectangular skew lattice is defined on  $L\times R$  by
$$(a,b)\wedge (c,d)=(a,d)=(c,d)\vee (a,b).$$

To within isomorphism, every rectangular skew lattice is a copy of some such algebra.

 Let ${\mathcal{D}}$ be the equivalence on $S$ given by $x\mathrel{\mathcal{D}} y$ if and only if $x\wedge y\wedge x=x$ and $y\wedge x\wedge y=y$. The Clifford-McLean theorem for bands \cite{H} extends to skew lattices \cite[1.4]{L3} in that  ${\mathcal{D}}$ is a congruence on $S$, the ${\mathcal{D}}$-classes are maximal rectangular subalgebras of $S$ and the quotient $S/{\mathcal{D}}$ is the maximal lattice quotient of $S$. Note that $\{0\}$ always forms a separate ${\mathcal D}$-class. If $S$ is an SBA then $S/{\mathcal{D}}$ is the maximal GBA quotient of $S$. 
 
 Throughout the paper, for $a\in S$ by $[a]$ we  denote the
 ${\mathcal{D}}$-class of $a$ and by $\pi\colon S\to S/{\mathcal D}$  the {\em canonical projection} map that takes $a$ to $[a]$. The assignment $S\mapsto S/{\mathcal D}$ may be extended to a functor from the category of SBAs to the category of GBAs which is a left adjoint to the inclusion functor in the reverse direction. Thus $S/{\mathcal D}$ is sometimes called the {\em commutative reflection} of $S$.

\subsection{Left-handed skew Boolean algebras}\label{sub:left-handed}

A skew lattice is called {\em left-handed} (resp. {\em right-handed}) if it satisfies the identities
$$x\wedge y\wedge x=x\wedge y \text{ and } x\vee y\vee x=y\vee x$$
$$(\text{resp. } x\wedge y\wedge x=y\wedge x \text{ and } x\vee y\vee x=x\vee y).$$

In a left-handed skew lattice the rectangular subalgebras are {\em left flat} meaning that $x\mathrel{\mathcal{D}} y$ if and only if $x\wedge y=x$ and $y\wedge x=y$.
A left-handed SBA $S$ can be characterised as an SBA where the normality axiom is replaced by the following stronger axiom

\begin{enumerate}
\item[(6$'$)] (left normality) $x\wedge y\wedge z=x\wedge z\wedge y$,
\end{enumerate}
and a dual axiom holds for right-handed SBAs. It follows that left-handed (as well as right-handed) SBAs form a variety of algebras. 

Let $S$ be a left-handed SBA, $a\in S$ and $\beta\in S/{\mathcal D}$ where $\beta\leq [a]$. There is a unique $b\in S$ such that $a\geq b$ and $[b] = \beta$.  Indeed, for all $c\in \beta$, left normality implies $a\geq a\wedge c$ where $a\wedge c\in\beta$. But given $c,d\in \beta$ we get $a\wedge c=a\wedge c\wedge d=a\wedge d\wedge c=a\wedge d$. This element $b$ is called the {\em restriction} of $a$ to $\beta$ and is denoted by $a|_{\beta}$. Note that $a\wedge b = a|_{[a] \wedge [b]}$ and $a\setminus b = a|_{[a]\setminus [b]}$ for any $a,b\in S$.

An SBA $S$ is called {\em primitive} if $S/{\mathcal D}={\bf 2}$ where ${\bf 2}=\{0,1\}$ is a two-element Boolean algebra. Equivalently, $S$ is primitive if and only if it has a unique non-zero ${\mathcal D}$-class. That is, it is just a rectangular SBA with a zero adjoined. Let $X$ be a set and $0\not\in X$. Putting $a\wedge b=a$, $a\vee b=b$ and $a\setminus b=0$ for any $a,b\in X$ defines on $X\cup \{0\}$ the (unique possible) structure of a primitive left-handed SBA with zero $0$ and non-zero ${\mathcal D}$-class $X$.    If $X=\{1,2,\dots, n\}$, where $n\geq 1$, the primitive left-handed SBA $X\cup \{0\}$ is denoted by ${\bf{(n+1)}}_L$. Just as any GBA can be embedded into a power of ${\bf 2}$, any left-handed SBA can be embedded into a power of ${\bf 3}_L$  \cite[Corollary 1.14]{L2}. See and   also \cite{Kud1} for a construction of a canonical such embedding. Moreover, just as any finite GBA is isomorphic to a finite power of ${\bf 2}$, any  left-handed SBA $S$ with $S/{\mathcal D}$ finite is isomorphic to a finite product of primitive left-handed SBAs \cite[Theorem 1.16]{L2}.

Let $(S;\wedge,\vee,\setminus,0)$ be a left-handed SBA. We define new operations $\wedge'$ and $\vee'$ on $S$ by
$a\wedge' b = b\wedge a$ and $a\vee' b= b\vee a$. Then $(S;\wedge',\vee',\setminus, 0)$ is right-handed. It is the {\em right-handed dual} of $S$.  In what follows, when discussing one-sided SBAs, we consider left-handed algebras, but upon dualization, similar definitions, results, etc. also hold for right-handed algebras.

The relations ${\mathcal{L}}$ and ${\mathcal{R}}$ on a skew lattice $S$ are given by
$$a\mathrel{\mathcal{L}} b \Leftrightarrow a \wedge b = a \text{ and } b \wedge a = b,$$
$$a\mathrel{\mathcal{R}}b \Leftrightarrow  a \wedge b = b  \text{ and } b \wedge a = a.$$
 The following extends Kimura's respective result for regular bands \cite{Kim}. 
 
\begin{theorem}[{Leech \cite[1.6]{L3}}] \label{th:decomp2} The relations ${\mathcal{L}}$ and ${\mathcal{R}}$ are congruences for any skew lattice $S$. Moreover, $S/{\mathcal{L}}$ is the maximal right-handed image of $S$, $S/{\mathcal{R}}$ is the maximal left-handed image of $S$, and the following diagram is a pullback:
\begin{center}
\begin{tikzpicture}[scale=0.9]
\matrix (m) [matrix of math nodes, row sep=3.3em, column sep=1.8em, text height=1.2ex, text depth=0.25ex]
{  S & & S/{\mathcal{R}} \\
      S/{\mathcal{L}} & & S/{\mathcal{D}} \\};
\path[->>] (m-1-1) edge (m-2-1);
\path[->>] (m-1-1) edge (m-1-3);
\path[->>] (m-2-1) edge  (m-2-3);
\path[->>] (m-1-3) edge  (m-2-3);
\end{tikzpicture}
\end{center}
\end{theorem}

Theorem \ref{th:decomp2} provides an effective and direct tool to extend results obtained for left-handed SBAs to general SBAs.
Thus all results obtained in Sections \ref{s:normal_forms} and \ref{s:free} of this paper for left-handed algebras, admit  extensions to general algebras. These extensions can be obtained similarly as in \cite{KL} where results on free SBAs are discussed and explicitly formulated also for right-handed and general algebras. We leave the details to an interested reader.

\subsection{Skew Boolean intersection algebras} A skew Boolean algebra has {\em (finite) intersections} if any finite set of its elements has the greatest lower bound with respect to the underlying partial order, called the {\em intersection} and denoted $\sqcap$.\footnote{This differs from the standard notation, but we need to reserve the symbols $\cap$ and $\cup$ to denote set intersection and union.} It is of course enough to require only that binary intersections exist.

If an SBA $(S; \wedge,\vee,\setminus, 0)$ has intersections, upon adding the intersection operation $\sqcap$ to the signature of the algebra, we get the algebra $(S;\wedge,\vee,\setminus, \sqcap, 0)$ called a {\em skew Boolean intersection algebra} or, briefly, an SBIA.  If $(S;\wedge,\vee,\setminus, \sqcap, 0)$ is an SBIA, then $(S; \wedge,\vee,\setminus, 0)$ is its {\em SBA-reduct}. According to our earlier convention, when this does not cause an ambiguity, we abbreviate an SBIA or a left-handed SBIA  $(S;\wedge,\vee,\setminus, \sqcap, 0)$ simply by $S$. By \cite[Proposition 2.6]{BL} we have:

\begin{proposition}
Let $(S; \wedge,\vee,\setminus, \sqcap, 0)$ be an algebra of type $(2,2,2,2,0)$. Then it is an SBIA if and only if $(S;\wedge,\vee,\setminus, 0)$ is an SBA, $(S,\sqcap)$ is a semilattice (meaning that $\sqcap$ is idempotent and commutative) and the following identities hold:
$$
x\sqcap (x\wedge y\wedge x)=x\wedge y\wedge x; \quad x\wedge (x\sqcap y) = x\sqcap y = (x\sqcap y)\wedge x.
$$
 \end{proposition}

 Thus SBIAs and left-handed SBIAs form varieties of algebras.
 
 Let $(S;\wedge,\vee,\setminus, \sqcap, 0)$ be an SBIA. The congruence  ${\mathcal D}$ on its SBA reduct $(S;\wedge,\vee,\setminus, 0)$ is in general no longer a congruence on $(S;\wedge,\vee,\setminus, \sqcap, 0)$, as  ${\mathcal D}$  is is not in general respected by $\sqcap$. For example, in ${\bf 3}_L$ we have $1 \mathrel{{\mathcal D}} 1$ and $1 \mathrel{{\mathcal D}} 2$ but $1=1\sqcap 1$ is not  ${\mathcal{D}}$-related with  $0 = 1\sqcap 2$. More generally, if $a,b\in S$ and $a \mathrel{\mathcal D} b$ then $a \mathrel{\mathcal D} (a\sqcap b)$ holds if and only if $a=b$. It follows that ${\mathcal D}$  is respected by $\sqcap$ if and only if $S$ is commutative in which case $\sqcap=\wedge$.
By $S/{\mathcal D}$ we denote the quotient over ${\mathcal D}$ of the SBA reduct of $S$, bearing in mind that in general $[a\sqcap b]\not=[a]\wedge [b]$.

\subsection{Some properties of  left-handed SBIAs}

An SBIA is {\em primitive} if its SBA reduct is a primitive SBA (note that any primitive SBA has intersections). By \cite[Theorem 3.5]{BL}, every SBIA can be embedded (as an SBIA) into a product of primitive SBIAs.  

A ${\mathcal D}$-class $D$ of an SBA $S$ is {\em atomic}, if $D$ is an {\em atom} of $S/{\mathcal D}$, that is, $D\neq\{0\}$ and if $a\leq D$ holds in $S/{\mathcal D}$ then either $a=D$ or $a=0$. We recall the following known structure result which extends the classical fact that a finite Boolean algebra is isomorphic to the powerset Boolean algebra of the set of its atoms.

\begin{proposition}\label{prop:structure}
Let $S$ be a left-handed SBA which has a finite number of ${\mathcal D}$-classes and let $D_1,\dots, D_m$ be the list of all atomic ${\mathcal D}$-classes of $S$. Then $S$ is isomorphic, as an SBIA, to the product $D_1^0\times \dots \times D_m^0$ where, for each $i=1,\dots, m$, $D_i^0$ is the only primitive left-handed SBA on the set $D_i\cup \{0\}$. 
\end{proposition}
\begin{proof}
We outline the proof, and the details can be readily recovered. To  $s\in S$ we assign $(a_1,\dots, a_m)\in D_1^0\times \dots \times D_m^0$ where, for each $k=1,\dots, m$, $a_k$ is either the atom in $D_k$ below $s$, if there is an atom in $D_k$ below $s$ (and then such an atom is unique), or $0$, otherwise. This assignment extends to an isomorphism between left-handed SBIAs. 
\end{proof}

Thus every finite left-handed SBIA is isomorphic to some ${\bf 2}^{m_2}\times {\bf 3}_L^{m_3}\times \dots \times {\bf (n+1)}_L^{m_{n+1}}$, where $m_i\geq 0$ for all $i$ under a convention that ${\bf n}_L^{0}\simeq \{0\}$.

The following results can be derived from \cite[Proposition 3.8]{BL} as well as from the duality theory of \cite{Kud}.

\begin{proposition} \label{prop:cong} Let $S$ be a finite SBIA and $\theta$ be a congruence on $S$.  Then the $\theta$-class $[0]_{\theta}$ of $0$ is either just $\{0\}$ or there are some atomic ${\mathcal D}$-classes, $D_1,\dots, D_k$, of $S$ such that $s\in [0]_{\theta}$ if and only if any atom below $s$ belongs to one of $D_1,\dots, D_k$. Furthermore, if $s\not\in [0]_{\theta}$ then $[s]_{\theta}=\{s\}$. 
\end{proposition}

\begin{proposition}\label{prop:10}
Let $S={\bf 2}^{m_2}\times {\bf 3}_L^{m_3}\times \dots \times {\bf (n+1)}_L^{m_{n+1}}$ and 
$T = {\bf 2}^{k_2}\times {\bf 3}_L^{k_3}\times \dots \times {\bf (n+1)}_L^{k_{n+1}}$, where 
$m_i, k_i\geq 0$ for each admissible $i$, be two finite left-handed SBIAs. Then:
\begin{enumerate}
\item $T$ is isomorphic to a quotient of $S$ (as an SBIA) if and only if $T$ is a direct factor of $S$, that is, $k_i\leq m_i$ for all $i$.
\item $T$ is isomorphic to the maximal commutative quotient of $S$ if and only if $m_2=k_2$ and $m_i=0$ for $i\geq 3$, that is, when $T\simeq {\bf 2}^{k_2}$.
\item The maximal commutative quotient of $S$ (as an SBIA) is isomorphic to the center of $S$. 
\end{enumerate}
\end{proposition}

\begin{proof} (1)  follows from Proposition \ref{prop:cong} and Proposition \ref{prop:structure} and its proof.
(2) is clear as the maximal commutative direct factor of $S$ is ${\bf 2}^{m_2}$.
(3) follows from (2) and the fact that the center of $S$ is isomorphic to ${\bf 2}^{m_2}$, see \cite[Theorem 1.7]{L3}.
\end{proof}

For future use, we need to record some identities holding in left-handed SBAs.

\begin{lemma}\label{lem:basic_sba} Let $S$ be a  left-handed SBA and  $x,y,z,t\in S$.
\begin{enumerate}
\begin{multicols}{2}
\item\label{sba1} $x\setminus (y\vee z)=x\setminus (z\vee y)$;
\item \label{sba2}  $(x\setminus y)\wedge z=(x\wedge z)\setminus y$;
\item \label{sba3} $x=(x\wedge y)\vee (x\setminus y)$;
\columnbreak
\item \label{sba4} $(x\setminus y)\setminus z=x\setminus (y\vee z)$;
\item \label{sba5} $(x\setminus y)\wedge (z\setminus t) = (x\wedge z)\setminus (y\vee t)$.
\end{multicols}
\end{enumerate}
\end{lemma}

\begin{proof} It is enough to verify that the identities hold in ${\bf 3}_L$ and apply the fact that any left-handed SBA can be embedded into a power of ${\bf 3}_L$.
\end{proof}

\begin{lemma} \label{lem:aux1}Let $S$ be an SBA. For any $x,y,z,s,t\in S$: if $x\setminus y=x\setminus z$ then $(x\wedge s)\setminus (y\vee t)= (x\wedge s)\setminus (z\vee t)$ and $(s\wedge x)\setminus (y\vee t)= (s\wedge x)\setminus (z\vee t)$.
\end{lemma}
\begin{proof}
Applying parts \eqref{sba4} and \eqref{sba2}  of Lemma \ref{lem:basic_sba}, we obtain:
$$
(x\wedge s)\setminus (y\vee t) = ((x\wedge s)\setminus y)\setminus t = ((x\setminus y )\wedge s)\setminus t.
$$
But then also $(x\wedge s)\setminus (z\vee t)=((x\setminus z )\wedge s)\setminus t$, so that the first equality follows.
The second equality is proved similarly.

Alternatively, as these are quasi-identities, it is enough to prove them for ${\bf 3}_L$, as in the previous proof.
\end{proof}

The following simple observation will be frequently used.

\begin{lemma}\label{lem:useful} Let $S$ be a left-handed SBA and $a,b\in S$.
If $[a]\leq [b]$ then
$a\setminus b=0$.
\end{lemma}

\begin{proof} Since $[a\setminus b]=[a]\setminus [b]$ is the zero of $S/{\mathcal D}$ then $a\setminus b=a|_{[a]\setminus [b]}=a|_{[0]}=0$. 
\end{proof}

In the following lemma we collect some identities which hold in left-handed SBIAs.

\begin{lemma}\label{lem:basic_sbia} Let $S$ be a left-handed SBIA and $x,y,z\in S$.
\begin{enumerate}
\begin{multicols}{2}
\item \label{sbia3} $(x\wedge y)\sqcap (z\wedge y)=(x\sqcap z)\wedge y$;
\item \label{sbia4} $(x\sqcap y)\setminus (x\sqcap z)=(x\sqcap y)\setminus (x\sqcap y\sqcap z)$;
\columnbreak
\item \label{sbia5} $(x\vee y)\sqcap z =(x\sqcap z)\vee  (y\sqcap z)$.
\end{multicols}
\end{enumerate}
\end{lemma}

\begin{proof} It is enough to verify the identities for primitive left-handed SBIAs. This is reduced to consideration of several cases, depending on if each pair of given elements $x,y,z$ has the same evaluation or not. For example, in the case where $x,y,z$ have pairwise distinct evaluations, the intersection of each pair is $0$. Then, for the first equality, both the left-hand side and the right-hand side equal $0$, etc.
\end{proof}

 Elements $a$ and $b$ of a left-handed SBA $S$ are called {\em compatible}, denoted  $a\sim b$, if $a\wedge b=b\wedge a$. Clearly, if $a,b\leq c$ for some $c\in S$ then $a\sim b$.
\begin{lemma}\label{lem:compat}\mbox{}
Let $S$ be a left-handed SBIA and $a,b\in S$.
Then $a\sim b$ if and only if $a\wedge b=a\sqcap b$.
\end{lemma}
\begin{proof} Note that $[a\sqcap b]\leq [a]\wedge [b]=[a\wedge b]$ so that $a\sqcap b\leq a\wedge b$, as both are restrictions of $a$. Thus $a\wedge b=a\sqcap b$ is equivalent to $a\wedge b\leq a\sqcap b$. Assume that $a\sim b$, that is, $a\wedge b=b\wedge a$. Then $a\wedge b\leq a,b$ so that $a\wedge b\leq a\sqcap b$.
Conversely, assume that $a\sqcap b=a\wedge b$. Then $a\wedge b\leq b$ so that $a\wedge b=b|_{[a]\wedge [b]}=b\wedge a$.
\end{proof}

\section{Elementary elements and normal forms}\label{s:normal_forms}

Throughout this section, $S$ is a fixed left-handed SBIA and $X$ is a fixed finite non-empty subset of~$S$.

\subsection{Set partitions}
Let $Y$ be a non-empty set. The set of all partitions of $Y$ into non-empty subsets is denoted by ${\mathcal P}(Y)$. If $\alpha\in {\mathcal P}(Y)$,  $Y$ is the {\em domain} of $\alpha$, denoted by ${\mathrm{dom}}(\alpha)$.
A partition  $\alpha\in {\mathcal P}(Y)$ can be looked at as an equivalence relation on $Y$, so that  $x\mathrel{\alpha} y$ means that $x$ and $y$ belong to the same block of $\alpha$.
If $\alpha\in {\mathcal P}(Y)$ has the blocks $A_1,\dots, A_k$, we write $\alpha=\{A_1,\dots, A_k\}$ and say that $k$ is the {\em rank} of $\alpha$, denoted ${\mathrm{rank}}(\alpha)$. If the set $Y$ is linearly ordered, we agree to order the blocks of $\alpha\in {\mathcal P}(Y)$ by their minimum elements and write the elements inside each block in the increasing order. Furthermore, we adopt  a standard convention to list elements of the blocks rather than the blocks themselves and separate elements by vertical lines. Thus, e.g., the partition $\{\{x_1\},\{x_2,x_3\},\{x_4,x_5\}\}$ of $\{x_1,x_2,x_3,x_4,x_5\}$ is denoted by
$x_1|x_2 x_3| x_4x_5$. We also write $A\in \alpha$ to indicate that $A$ is a block of $\alpha$.

Our focus in this paper will be partitions of subsets of a given set. If $\alpha$ is a partition of a subset of $X$, we sometimes call the pair $(X,\alpha)$ just a {\em partition}.  Let  $Z\subseteq Y$, $Z\neq\varnothing$, and  $\alpha, \beta$ be partitions of non-empty subsets of $Z$ and $Y$, respectively.  We say that $(Y,\beta)$ {\em contains} $(Z,\alpha)$ and write $(Z,\alpha)\preceq (Y,\beta)$ if $${\mathrm{dom}}(\alpha)\subseteq {\mathrm{dom}}(\beta)\subseteq {\mathrm{dom}}(\alpha) \cup (Y\setminus Z)$$ and
for any $x,y\in {\mathrm{dom}}(\alpha)$: $x\mathrel{\alpha} y$ if and only if $x\mathrel{\beta} y$.

\begin{example} {\em Let $Y=\{x_1,x_2,x_3,x_4\}$, $Z=\{x_1,x_2,x_3\}$, $\alpha=x_1|x_2$,  $\beta=x_1|x_2x_4$
$\gamma=x_1x_3|x_2x_4$.  Then $(Y,\beta)$ contains $(Z,\alpha)$, but $(Y,\gamma)$ does  not as ${\mathrm{dom}}(\gamma)\not\subseteq {\mathrm{dom}}(\alpha)\cup (Y\setminus Z)$.}
\end{example}

Let $Y,Z, \alpha,\beta$ be as above and let  $(Z,\alpha)\preceq (Y,\beta)$. Each block $A$ of $\alpha$ is contained in a (unique) block of $\beta$ called the block {\em induced} by $A$ and denoted by $A\!\uparrow_{\alpha}^{\beta}$. Conversely, for each block $B$ of $\beta$ its intersection with ${\mathrm{dom}}(\alpha)$ is a block of $\alpha$  called the {\em restriction} of $B$ and denoted by $B\!\downarrow^{\beta}_{\alpha}$.

We call a partition $\alpha$ of a non-empty subset of $X$ {\em pointed} if some block $A\in\alpha$ is marked. We denote such a pointed partition by $(X,\alpha,A)$.

\subsection{Elements determined by pointed partitions}

Let $A=\{a_1,\dots, a_n\}$ be a finite non-empty subset of $S$.

\begin{itemize}
\item By $\sqcap A$ we denote the element $a_1\sqcap a_2\sqcap\dots \sqcap a_n$. This  is well-defined since the operation $\sqcap$ is commutative.

\item For $a\in A$, by $a\wedge (\wedge A)$ we denote the element $a\wedge a_1\wedge \dots\wedge a_n$. This is well-defined since the operation $\wedge$ is left normal.

\item For $s\in S$, by $s\setminus (\vee A)$ we denote the element $s\setminus (a_1\vee \dots \vee a_n)$. This is well-defined by part \eqref{sba1} of Lemma \ref{lem:basic_sba}. We also put $s\setminus (\vee \varnothing) = s$.
\end{itemize}

We turn to one of the main constructions of the paper. Let  $(X,\alpha,A)$ be a pointed partition of a non-empty subset of $X$.  Suppose that $\alpha=\{A_1,\dots, A_k\}$ (thus $A=A_i$ for some $i$) and $Y={\mathrm{dom}}(\alpha)$. We define the element
\begin{equation}\label{eq:pq}
e(X,\alpha, A) =p\setminus (\vee Q), \, \text{ where }
\end{equation}
\begin{equation*}\label{eq:def_p}
p= (\sqcap A)\wedge (\wedge \{\sqcap A_i \colon 1\leq i\leq k\})=(\sqcap A)\wedge (\sqcap A_1)\wedge\dots \wedge (\sqcap A_k) \text{ and }
\end{equation*}
\begin{equation*}\label{eq:def_q}
Q=(X\setminus Y)\cup \{\sqcap (A_i \cup  A_j)\colon 1\leq i<j\leq k\}.
\end{equation*}

For the reason, which will become clear later, we  call the elements $e(X,\alpha,A)$ the {\em elementary elements} over $X$.  We call the partition $(X,\alpha)$ the {\em support} of $e(X,\alpha, A)$.

\begin{remark} \label{rem:1} {\em In the special case where ${\mathrm{rank}}(\alpha)=1$, $Q$ equals $X\setminus Y$. In particular, if ${\mathrm{rank}}(\alpha)=1$ and $Y=X$,  $Q=\varnothing$, so that $e(X,\alpha,A)=p\setminus (\vee \varnothing)=p$.} \end{remark}

\begin{example}
{\em Let $X=\{x_1,x_2,x_3,x_4,x_5\}$,  $\alpha=x_1x_3|x_4$, $\beta=x_2|x_3x_4|x_5$ and $\gamma=x_1x_2x_3x_4x_5$. Then:
$$\begin{array}{l}
e(X,\alpha, \{x_4\})=(x_4\wedge (x_1\sqcap x_3))\setminus (x_2\vee x_5\vee (x_1\sqcap x_3\sqcap x_4));\\
e(X,\alpha, \{x_1,x_3\})=((x_1\sqcap x_3) \wedge x_4)\setminus (x_2\vee x_5\vee (x_1\sqcap x_3\sqcap x_4));\\
e(X,\beta, \{x_3,x_4\})=((x_3\sqcap x_4)\wedge  x_2\wedge x_5)\setminus (x_1\vee (x_2 \sqcap x_3 \sqcap x_4)\vee  (x_2 \sqcap x_5)\vee (x_3\sqcap x_4 \sqcap x_5));\\
e(X,\gamma, \{x_1,x_2,x_3,x_4,x_5\})=x_1\sqcap x_2\sqcap x_3\sqcap x_4\sqcap x_5.
\end{array}
$$}
\end{example}

\begin{proposition}\label{prop:crucial} Let  $(X,\alpha,A)$ and $(X,\beta,B)$ be pointed partitions of non-empty subsets of $X$. Then
\begin{enumerate}
\item \label{i:1}
$$
e(X,\alpha,A)\wedge e(X,\beta,B)=\left\lbrace\begin{array}{ll}e(X,\alpha, A), & \text{if } \alpha=\beta;\\
0, & \text{otherwise.}\end{array}\right.
$$

\item \label{i:2}
$$
e(X,\alpha,A)\sqcap e(X,\beta,B)=\left\lbrace\begin{array}{ll}e(X,\alpha, A), & \text{if } \alpha=\beta \text{ and } A=B;\\
0, & \text{otherwise.}\end{array}\right.
$$
\item\label{i:3}
$$
e(X,\alpha,A)\setminus e(X,\beta,B)=\left\lbrace\begin{array}{ll}e(X,\alpha, A), & \text{if } \alpha\neq \beta;\\
0, & \text{otherwise.}\end{array}\right.
$$
\end{enumerate}
\end{proposition}

\begin{proof}
\eqref{i:1} Suppose  $e(X,\alpha,A)=p_1\setminus (\vee Q_1)$ and $e(X,\beta,B)=p_2\setminus (\vee Q_2)$, see \eqref{eq:pq}.
By part \eqref{sba5} of Lemma \ref{lem:basic_sba},  $e(X,\alpha, A)\wedge e(X,\beta,B)= (p_1\wedge p_2)\setminus \vee (Q_1\cup Q_2)$. We consider two possible cases.

{\bf Case 1.} Suppose first  $\alpha=\beta$.  Then $Q_1=Q_2$. By left normality,  $p_1\wedge p_2=p_1$. The needed equality follows.

{\bf Case 2.} Suppose now  $\alpha\neq \beta$. We divide this case into two subcases.

{\bf Subcase 2.1.} Suppose ${\mathrm{dom}}(\alpha)\neq {\mathrm{dom}}(\beta)$. Since $x\wedge y=0$ implies $y\wedge x=0$, we may assume, without loss of generality, that there is $x\in {\mathrm{dom}}(\alpha)\setminus {\mathrm{dom}}(\beta)$. Let $A_x$ be the block of $\alpha$ containing $x$. Then $\sqcap A_x\leq x$ and thus $[p_1]\leq [x]$. On the other hand, we have $[\vee Q_2]\geq [x]$ by the construction of $Q_2$. Applying Lemma \ref{lem:useful}, it follows that $e(X,\alpha, A)\wedge e(X,\beta,B) =0$.

{\bf Subcase 2.2.} Suppose ${\mathrm{dom}}(\alpha)={\mathrm{dom}}(\beta)$. Since $\alpha\neq \beta$, there are some $x,y\in {\mathrm{dom}}(\alpha)$ such that $x$ and $y$ are in the same block of $\alpha$ but in different blocks of $\beta$, or dually. We assume the former (and the dual case is similar). By $C_{xy}$ we denote the block of $\alpha$ containing $x$ and $y$, and by $C_x$ and $C_y$ the blocks of $\beta$ containing $x$ and $y$, respectively. Then $p_1\wedge p_2$ equals  $(\sqcap C_1)\wedge (\sqcap C_2) \wedge \dots \wedge (\sqcap C_l)$ where the first several $C_i$ are precisely all the blocks of $\alpha$ which are followed by all of the blocks of $\beta$. In particular, $C_{xy}$, $C_x$ and $C_y$ are among these blocks $C_i$. By left normality, $p_1\wedge p_2=(\sqcap C_1)\wedge (\sqcap C_{\sigma(2)}) \wedge \dots \wedge (\sqcap C_{\sigma(l)})$ where $\sigma$ is any permutation of the set $\{2,\dots, l\}$. We may thus assume that the blocks $C_x$, $C_{xy}$ and $C_y$ are some $C_k, C_{k+1}$ and $C_{k+2}$, respectively (in fact, $k$ can be even chosen equal $1$ or $2$). Since $\sqcap C_x,\sqcap C_{xy}\leq x$, applying Lemma \ref{lem:compat} we get
$$
(\sqcap C_x)\wedge (\sqcap C_{xy})=(\sqcap C_x)\sqcap (\sqcap C_{xy})=\sqcap (C_x\cup C_{xy}).
$$
Similarly, $(\sqcap C_{xy})\wedge (\sqcap C_y) =  \sqcap (C_y\cup C_{xy})$. Hence,
$$
(\sqcap C_x)\wedge (\sqcap C_{xy}) \wedge (\sqcap C_y)= \sqcap (C_x\cup C_{xy}\cup C_y).
$$
It follows that $[p_1\wedge p_2]\leq [\sqcap (C_x\cup C_{xy}\cup C_y)]\leq [\sqcap (C_x \cup C_y)]$.
But since one of the elements of $Q_2$ is  $\sqcap (C_x \cup C_y)$,  $[\vee (Q_1\cup Q_2)]\geq [\vee Q_2]\geq [\sqcap (C_x \cup C_y)]$. By Lemma \ref{lem:useful}, we obtain $(p_1\wedge p_2)\setminus (\vee (Q_1\cup Q_2))=0$. This finishes the proof of part (1).

\eqref{i:2}  In the case where $\alpha\neq \beta$, the needed equality follows from \eqref{i:1}, because $x\sqcap y \leq x\wedge y$. Also, the case $\alpha=\beta$ and $A=B$ is obvious. So it is enough to consider only the case where $\alpha=\beta$ and $A\neq B$. Assume that $\alpha=\beta=\{C_1,\dots, C_m\}$. Reindexing the blocks, if needed, we may assume that $A=C_1$ and $B=C_2$. By left normality, we have $p_1=(\sqcap C_1)\wedge (\sqcap C_2) \wedge (\sqcap C_3) \wedge \dots \wedge (\sqcap C_m)$ and $p_2=(\sqcap C_2)\wedge (\sqcap C_1) \wedge (\sqcap C_3) \wedge \dots \wedge (\sqcap C_m)$. Applying part \eqref{sbia3} of Lemma \ref{lem:basic_sbia}, we obtain
\begin{align*}
p_1\sqcap p_2 = & (\sqcap (C_1\cup C_2))\wedge (\sqcap C_1) \wedge (\sqcap C_2) \wedge \dots \wedge (\sqcap C_m) \\
= & (\sqcap (C_1\cup C_2))\wedge (\sqcap C_3) \wedge \dots \wedge (\sqcap C_m).
\end{align*}
Therefore, $[p_1\sqcap p_2]\leq  [(\sqcap (C_1\cup C_2))]$. On the other hand, $\sqcap (C_1\cup C_2)$ belongs to $Q_1$, so that $[\vee Q_1]\geq [(\sqcap (C_1\cup C_2))]$. Thus, in view of  Lemma \ref{lem:useful}, $e(X,\alpha,A)\sqcap  e(X,\alpha,B)= 0$.
\end{proof}

\begin{corollary}\label{cor:d} Let $e(X,\alpha, A)$ and $e(X,\beta, B)$ be two non-zero elementary elements.
Then
\begin{enumerate}
\item \label{dec1} $e(X,\alpha, A)=e(X,\beta, B)$ if and only if $\alpha=\beta$ and $A=B$.
\item \label{dec2} $e(X,\alpha,A)\mathrel{\mathcal{D}} e(X,\beta,B)$ if and only if $\alpha=\beta$.
\end{enumerate}
\end{corollary}
\begin{proof} (1) Assume $e(X,\alpha, A)=e(X,\beta, B)$ but $\alpha=\beta$ and $A=B$ does not hold. Then by part~\eqref{i:2} of Proposition \ref{prop:crucial} we get
$0=e(X,\alpha, A)\sqcap e(X,\beta, B)=e(X,\alpha, A)\sqcap e(X,\alpha, A)=e(X,\alpha, A)$, which is impossible by our assumption.

(2) By part \eqref{i:1} of Proposition \ref{prop:crucial}  $e(X,\alpha, A)\mathrel{\mathcal D}e(X,\beta,B)$ is equivalent to $e(X,\alpha,A)\wedge e(X,\beta,B)=e(X,\alpha,A)$.
\end{proof}

The next corollary witnesses certain `rigidity' of the behaviour of elementary elements, particularly of those which are  ${\mathcal D}$-related.

\begin{corollary}\label{cor:12}\mbox{}
\begin{enumerate}
\item \label{iii1} $e(X,\alpha, A)=0$ if and only if $e(X,\alpha,B)=0$ for all $B\in \alpha$.
\item \label{iii3} Suppose $e(X,\alpha, A)\neq 0$. Then  $e(X,\alpha, B)\neq 0$ for all $B\in \alpha$. Furthermore,  for any $B,C\in \alpha$, $e(X,\alpha, B)=e(X,\alpha, C)$ if and only if $B=C$. Consequently, the cardinality of the ${\mathcal D}$-class $[e(X,\alpha, A)]$ equals ${\mathrm{rank}}(\alpha)$. 
\item \label{iii4} $e(X,\alpha, A)=e(X,\beta, B)$ if and only if either $e(X,\alpha,A)=e(X,\beta,B)=0$, or, otherwise, $\alpha=\beta$ and $A=B$.
\end{enumerate}
\end{corollary}

\begin{proof}
\eqref{iii1} If $e(X,\alpha, A)=0$ then $e(X,\alpha,B)=e(X,\alpha,B)\wedge e(X,\alpha, A)=e(X,\alpha,B)\wedge 0=0$ for any $B\in \alpha$.

\eqref{iii3} This follows from part (1) above and part \eqref{dec1} of Corollary \ref{cor:d}.

\eqref{iii4} follows from \eqref{iii1} and  \eqref{iii3}.
\end{proof}

\subsection{The Decomposition Rule}

Let $e_1,\dots, e_n\in S$ and assume that $e_i\wedge e_j=0$ (and thus $e_j\wedge e_i=0$) for each $1\leq i < j\leq n$.
Then these elements are said to form an {\em orthogonal family} and the join $e_1\vee\dots\vee e_n$ is called an {\em orthogonal join}. Since $e_i\vee e_j = e_j\vee e_i$ for each $1\leq i < j\leq n$,  for every permutation $\sigma$ of $\{1,2,\dots, n\}$,  $e_{\sigma(a)}\vee \dots \vee e_{\sigma(n)}=e_1\vee\dots\vee e_n$.  It is thus correct to denote the orthogonal join $e_1\vee\dots\vee e_n$ by $\vee \{e_1,\dots,e_n\}$.

The following is one of the main result of the paper.

\begin{theorem}[Decomposition Rule] \label{th:branching} Let $S$ be a left-handed SBIA, $X,Y$  be finite non-empty subsets of $S$  such that $X\subseteq Y$. Let, further, $\alpha$ be a partition of a non-empty subset of $X$ and $A\in\alpha$.
Then 
\begin{equation}\label{eq:branching_d}
e(X,\alpha, A) = \vee\{e\left(Y,\beta, A\!\uparrow_{\alpha}^{\beta}\right)\colon (X,\alpha)\preceq (Y,\beta)\},
\end{equation}
the latter join  being orthogonal.  
\end{theorem}

\begin{proof}   That the family $\{e\left(Y,\beta, A\!\uparrow_{\alpha}^{\beta}\right)\colon (X,\alpha)\preceq (Y,\beta)\}$ is orthogonal follows from part~\eqref{i:1} of Proposition~\ref{prop:crucial}. We turn to proving equality \eqref{eq:branching_d}.

Suppose first that $|Y\setminus X|=1$. Then $Y=X\cup\{t\}$ where $t\not\in X$. Assume that $\alpha=\{A_1,\dots, A_k\}$. The partitions of subsets of $Y$ which contain $(X,\alpha)$ are $(Y,\alpha)$ and the  $k+1$ partitions $(Y,\beta_{i})$, $1\leq i\leq k+1$, with domain ${\mathrm{dom}}(\alpha)\cup \{t\}$ where
$$\beta_i = \{A_1,\dots, A_{i-1}, A_i\cup \{t\}, A_{i+1},\dots, A_k\} \text{ for } 1\leq i\leq k \text{ and }
$$
$$
\beta_{k+1}=\{A_1,\dots, A_k,\{t\}\}.
$$

If $A=A_j$ then
$$A\!\uparrow_{\alpha}^{\beta_i}=\left\lbrace \begin{array}{ll} A & \text{if } i\neq j;\\
A\cup\{t\}& \text{otherwise}.\end{array}\right.
$$

Denote the $(p\setminus (\vee Q))$-form of the element $e(Y,\beta_i,A\!\uparrow_{\alpha}^{\beta_i})$, where $1\leq i\leq k+1$, by
$p_i\setminus (\vee Q_i)$. We also write $(p\setminus (\vee Q))$ for the $(p\setminus (\vee Q))$-form of $e(X,\alpha,A)$.
We divide the following considerations into several steps.

{\bf Step 1.} By part \eqref{sba3} of Lemma \ref{lem:basic_sba},
\begin{equation}\label{eq:step1.1}
e(X,\alpha, A)=(e(X,\alpha, A)\setminus t)\vee  (e(X,\alpha, A)\wedge t)
\end{equation}
and by part \eqref{sba4} of Lemma \ref{lem:basic_sba},  $e(X,\alpha, A)\setminus t=e(Y,\alpha,A)$. Thus it remains to show that
\begin{equation}\label{eq:step1.2}
e(X,\alpha, A)\wedge t = \vee_{i=1}^{k+1} e(Y,\beta_i, A\!\uparrow_{\alpha}^{\beta_i}).
\end{equation}

{\bf Step 2.}
By part \eqref{sba2} of Lemma \ref{lem:basic_sba},
\begin{equation}\label{eq:step2.1}
e(X,\alpha, A)\wedge t=(p\wedge t)\setminus (\vee Q).
\end{equation}
Put $R=\{\sqcap(A_1\cup \{t\}), \dots, \sqcap(A_k\cup\{t\})\}$. Since, by part \eqref{sba2} of Lemma \ref{lem:basic_sba},
\begin{equation}\label{eq:step2.1a}(p\wedge t)\setminus (\vee (Q\cup R))=\big((p\wedge t)\setminus (\vee Q)\big)\setminus (\vee R),
\end{equation}
applying part \eqref{sba3} of the same lemma, we obtain
\begin{equation}\label{eq:step2.2}
(p\wedge t)\setminus (\vee Q)=\big((p\wedge t)\setminus (\vee (Q\cup R))\big)\vee \big(((p\wedge t)\setminus (\vee Q))\wedge (\vee R)\big).
\end{equation}
Observe that $Q\cup R= Q_{k+1}$, so that
\begin{equation}\label{eq:step2.3}
(p\wedge t)\setminus (\vee (Q\cup R))=e(Y,\beta_{k+1},A).
\end{equation}
So we are left to show that
\begin{equation}\label{eq:step2.4}
\left((p\wedge t)\setminus (\vee Q)\right)\wedge (\vee R) = \vee_{i=1}^{k} e(Y,\beta_i, A\!\uparrow_{\alpha}^{\beta_i}).
\end{equation}

{\bf Step 3.}
By distributivity and the definition of $R$, it is enough to show that, for each $i=1,\dots, k$,
\begin{equation}\label{eq:step3.1}
((p\wedge t)\setminus (\vee Q))\wedge (\sqcap(A_i\cup \{t\}))=e(Y,\beta_i, A\!\uparrow_{\alpha}^{\beta_i}).
\end{equation}
By part \eqref{sba2} of Lemma \ref{lem:basic_sba}, the left-hand side of the  expression above equals
\begin{equation}\label{eq:step3.2}
\big(p\wedge t\wedge (\sqcap(A_i\cup \{t\}))\big)\setminus (\vee Q).
\end{equation}
But since $A_i\subseteq A_i\cup \{t\}$ and $\{t\}\subseteq A_i\cup \{t\}$ and also $x\wedge y=y\wedge x =x$ whenever $x\leq y$, we obtain
\begin{equation}\label{eq:step3.3}
p\wedge t\wedge (\sqcap(A_i\cup \{t\}))=p_i.
\end{equation}
It remains to show that
\begin{equation}\label{eq:remains1}
p_i\setminus (\vee Q)=p_i\setminus (\vee Q_i).
\end{equation}

{\bf Step 4.}
 Note that the set $Q_i$ can be obtained from $Q$ by replacing each $\sqcap (A_i\cup A_j)$, where $i\neq j$, by $\sqcap (A_i\cup A_j\cup \{t\})$.
By part \eqref{sbia4} of Lemma \ref{lem:basic_sbia},
\begin{equation}\label{eq:step4.1}
(\sqcap(A_i\cup \{t\}))\setminus (\sqcap (A_i\cup A_j)) = (\sqcap(A_i\cup \{t\}))\setminus (\sqcap (A_i\cup A_j\cup \{t\}))
\end{equation}
and thus, after $k-1$ applications of Lemma \ref{lem:aux1}, \eqref{eq:remains1} follows.

Consider now the case where $|Y\setminus X|>1$. It is enough to assume that $Y=X\cup \{t,s\}$ where $t,s\not\in X$.
Let $\alpha$ be a partition of a non-empty subset of $X$ and $A\in\alpha$.  By \eqref{eq:branching_d},
$$
e(X,\alpha,A)=\vee \{e(X\cup \{t\}, \beta, A\!\uparrow_{\alpha}^{\beta})\colon (X,\alpha)\preceq (X\cup \{t\},\beta)\}$$
and for each $(X\cup \{t\},\beta)$ satisfying  $(X,\alpha)\preceq (X\cup \{t\},\beta)$, again by \eqref{eq:branching_d},
$$
e(X\cup \{t\}, \beta, A\!\uparrow_{\alpha}^{\beta}) = \vee\{e(Y, \gamma, (A\!\uparrow_{\alpha}^{\beta})\!\uparrow_{\beta}^{\gamma})\colon  (X\cup \{t\},\beta)\preceq (Y,\gamma)\}.
$$
But for any $(Y,\gamma)$  satisfying  $(X,\alpha)\preceq (Y,\gamma)$ there a unique partition $\beta$ of a subset of $X\cup \{t\}$ satisfying $(X,\alpha)\preceq (X\cup \{t\},\beta)\preceq (Y,\gamma)$: this is the restriction of $\gamma$ to $X\cup \{t\}$. Furthermore, $A\!\uparrow_{\alpha}^{\gamma}=(A\!\uparrow_{\alpha}^{\beta})\!\uparrow_{\beta}^{\gamma}$. Therefore
$$
e(X,\alpha,A)=\vee \{e(Y, \gamma, A\!\uparrow_{\alpha}^{\gamma})\colon (X,\alpha)\preceq (Y,\gamma)\},
$$
as required.
\end{proof}

We illustrate the proof of Theorem \ref{th:branching} by an example.

\begin{example}{\em
Let $X=\{x_1,x_2\}$ and $Y=\{x_1,x_2,x_3\}$. Let $\alpha=x_1|x_2$ and $\{x_1\}$ be the marked block. Following the proof of Theorem \ref{th:branching}, we show that
\begin{equation}\label{eq:aim}
e(X,\alpha, \{x_1\}) = e(Y,\alpha, \{x_1\})\vee e(Y,\beta_1, \{x_1,x_3\})\vee e(Y,\beta_2, \{x_1\})\vee e(Y,\beta_3, \{x_1\}),
\end{equation}
where $\beta_1=x_1x_3|x_2$, $\beta_2=x_1|x_2x_3$ and $\beta_3=x_1|x_2|x_3$. Note that
$$
\begin{array}{lcl}
e(X,\alpha, \{x_1\})& = & (x_1\wedge x_2)\setminus (x_1\sqcap x_2);\\
e(Y,\alpha, \{x_1\})& = & (x_1\wedge x_2)\setminus ((x_1\sqcap x_2)\vee x_3);\\
e(Y,\beta_1, \{x_1,x_3\}) & = & ((x_1\sqcap x_3)\wedge x_2)\setminus (x_1\sqcap x_2\sqcap x_3);\\
e(Y,\beta_2, \{x_1\})& = & (x_1\wedge (x_2\sqcap x_3))\setminus (x_1\sqcap x_2\sqcap x_3);\\
e(Y,\beta_3, \{x_1\})& = &(x_1\wedge x_2\wedge x_3)\setminus ((x_1\sqcap x_2)\vee (x_1\sqcap x_3)\vee (x_2\sqcap x_3)).
\end{array}
$$

{\bf Step 1.} Equation \eqref{eq:step1.1} is written as
$$e(X,\alpha, \{x_1\})=(e(X,\alpha, \{x_1\})\setminus x_3) \vee  (e(X,\alpha, \{x_1\})\wedge x_3).$$
Observe that $e(X,\alpha, \{x_1\})\setminus x_3 = e(Y,\alpha, \{x_1\})$. So we are left to show \eqref{eq:step1.2}, which in our case is
$$e(X,\alpha, \{x_1\})\wedge x_3 = e(Y,\beta_1, \{x_1,x_3\})\vee e(Y,\beta_2, \{x_1\})\vee e(Y,\beta_3, \{x_1\}).$$

{\bf Step 2.} Equation \eqref{eq:step2.1} is in our case
$$e(X,\alpha, \{x_1\})\wedge x_3 = (x_1\wedge x_2\wedge x_3)\setminus (x_1\sqcap x_2).$$
We set $R=\{x_1\sqcap x_3, x_2\sqcap x_3\}$. We then write \eqref{eq:step2.1a}:
$$
 (x_1\wedge x_2\wedge x_3)\setminus ((x_1\sqcap x_2)\vee (x_1\sqcap x_3)\vee (x_2\sqcap x_3))=(x_1\wedge x_2\wedge x_3)\setminus (x_1\sqcap x_2))\setminus  ((x_1\sqcap x_2)\vee (x_2\sqcap x_3)).
$$
Equation \eqref{eq:step2.2} is in our case
\begin{multline*}
 (x_1\wedge x_2\wedge x_3)\setminus (x_1\sqcap x_2) = \\ \Big((x_1\wedge x_2\wedge x_3)\setminus \big((x_1\sqcap x_2)\vee (x_1\sqcap x_2)\vee (x_2\sqcap x_3)\big)\Big) \vee \\ \Big(\big((x_1\wedge x_2\wedge x_3)\setminus (x_1\sqcap x_2)\big)\wedge \big((x_1\sqcap x_3)\vee (x_2\sqcap x_3)\big)\Big).
\end{multline*}
Equation \eqref{eq:step2.3} is in our case
$$
(x_1\wedge x_2\wedge x_3)\setminus \big((x_1\sqcap x_2)\vee (x_1\sqcap x_2)\vee (x_2\sqcap x_3)\big)=e(Y,\beta_3, \{x_1\}),
$$
so that we are left to show equation \eqref{eq:step2.4}:
$$
\big((x_1\wedge x_2\wedge x_3)\setminus (x_1\sqcap x_2)\big)\wedge \big((x_1\sqcap x_3)\vee (x_2\sqcap x_3)\big)=e(Y,\beta_1, \{x_1,x_3\}) \vee e(Y,\beta_2, \{x_1\}).
$$

{\bf Step 3.} In this step we observe that it is enough to show equalities \eqref{eq:step3.1} which in our case are
$$
\big((x_1\wedge x_2\wedge x_3)\setminus (x_1\sqcap x_2)\big)\wedge (x_1\sqcap x_3) =  ((x_1\sqcap x_3)\wedge x_2)\setminus (x_1\sqcap x_2\sqcap x_3) \text{ and }
$$
$$
\big((x_1\wedge x_2\wedge x_3)\setminus (x_1\sqcap x_2)\big)\wedge (x_2\sqcap x_3) =  (x_1\wedge (x_2\sqcap x_3))\setminus (x_1\sqcap x_2\sqcap x_3).
$$
We show the first equality, the second one being similar. The expression equal to its left-hand side, given in \eqref{eq:step3.2}, is $\big((x_1\wedge x_2\wedge x_3)\wedge (x_1\sqcap x_3)\big) \setminus (x_1\sqcap x_2).$
 In \eqref{eq:step3.3} we observe that $(x_1\wedge x_2\wedge x_3)\wedge (x_1\sqcap x_3)=(x_1\sqcap x_3)\wedge x_2$. It remains to show \eqref{eq:remains1} which in our case is
$$
\big((x_1\sqcap x_3)\wedge x_2\big)\setminus (x_1\sqcap x_2) =  \big((x_1\sqcap x_3)\wedge x_2\big)\setminus (x_1\sqcap x_2\sqcap x_3).
$$
{\bf Step 4.} Finally, in this step we notice that the above equality holds due to equality \eqref{eq:step4.1} which tells us that $(x_1\sqcap x_3)\setminus (x_1\sqcap x_2) =  (x_1\sqcap x_3)\setminus (x_1\sqcap x_2\sqcap x_3).$  Equation \eqref{eq:aim} is verified.
}
\end{example}

The following is an important consequence of the Decomposition Rule.

\begin{corollary}\label{cor:x} Let $x\in S$ and $Y\subseteq S$ be a finite subset containing $x$.
Then
\begin{equation}\label{eq:branching_x}
x=\vee\left\{e(Y,\beta, \{x\}\!\uparrow_{\{\{x\}\}}^{\beta})\colon x\in  \mathrm{dom}(\beta)\right\}.
\end{equation}
\end{corollary}
\begin{proof} It is enough to observe that a partition $(Y,\beta)$ contains the partition $(\{x\}, \{\{x\}\})$  if and only if $x\in {\mathrm{dom}}(\beta)$.
\end{proof}

\begin{example} Let $x_1,x_2,x_3\in S$. Then
\begin{align*}
x_1 = & \,\, (x_1\setminus (x_2\vee x_3))  \vee
 ((x_1\sqcap x_2)\setminus x_3) \vee \\
& ((x_1\wedge x_2)\setminus ((x_1\sqcap x_2)\vee x_3))\vee
 ((x_1\sqcap x_3)\setminus x_2) \vee \\
& ((x_1\wedge x_3)\setminus ((x_1\sqcap x_3)\vee x_2))\vee
 ((x_1\sqcap x_2\sqcap x_3)) \vee \\
& (((x_1\sqcap x_2)\wedge x_3)\setminus (x_1\sqcap x_2\sqcap x_3)) \vee
  (((x_1\sqcap x_3)\wedge x_2)\setminus (x_1\sqcap x_2\sqcap x_3)) \vee \\
&  ((x_1\wedge (x_2\sqcap x_3))\setminus (x_1\sqcap x_2\sqcap x_3)) \vee
 ((x_1\wedge x_2\wedge x_3)\setminus ((x_1\sqcap x_2)\vee (x_1\sqcap x_3)\vee (x_2\sqcap x_3))).
\end{align*}
\end{example}

By Corollary \ref{cor:d}, the ${\mathcal D}$-class $[e(X,\alpha, A)]$ of $e(X,\alpha, A)$ does not depend on $A$. We thus denote it by $e(X,\alpha)$.  The following is a consequence of Theorem \ref{th:branching}, since ${\mathcal D}$ is a congruence on the SBA reduct of $S$.

\begin{corollary}[Decomposition Rule for $S/{\mathcal D}$]\label{cor:21}
Let $S$ be a left-handed SBIA and $X\subseteq Y$ be  finite non-empty subsets of $S$. Let, further, $\alpha$ be a partition of a non-empty subset of $X$. Then, in $S/{\mathcal D}$,

\begin{equation}\label{eq:branching1}
e(X,\alpha)=\vee\{e(Y,\beta)\colon (X,\alpha)\preceq (Y,\beta)\},
\end{equation}
the latter join being orthogonal.
\end{corollary}

\subsection{Normal forms}  \label{subs:normal_forms}   Let ${\mathcal E}$ be a family of elementary elements over $X$. It is {\em admissible} if all elements in ${\mathcal E}$ are non-zero and also $e(X,\alpha,A), e(X,\alpha, B)\in {\mathcal E}$ implies that $A=B$, that is, ${\mathcal E}$ contains at most one elementary element with support $(X,\alpha)$ for each partition $(X,\alpha)$.
 By part \eqref{i:1} of Proposition \ref{prop:crucial} ${\mathcal E}$ is admissible of and only if all its elements are non-zero and it is orthogonal. A {\em normal form} {\em over} $X$ is an element which can be written as
$\vee {\mathcal E}$,
where  ${\mathcal E}$ is an admissible family of elementary elements over $X$. 
The elements of ${\mathcal E}$ are the {\em clauses} of the normal form. We take a convention that $0$ is a normal form, corresponding to the empty family ${\mathcal E}$.

\begin{theorem}[Normal Forms]\label{th:normal_forms}
Let ${\mathcal E}$ and ${\mathcal F}$ be two admissible families of elementary elements over $X$ and let $e= \vee {\mathcal E}$ and $f = \vee {\mathcal F}$ be the respective normal forms over $X$. Then 
\begin{enumerate}
\item \label{ii1} Any subfamily of ${\mathcal E}$ is admissible as is the family $\{e\in  {\mathcal{E}}\colon \pi(e)\not\in \pi({\mathcal{F}})\}\cup {\mathcal F}$. Moreover, the elements $e\vee f$, $e\wedge f$, $e\sqcap f$ and $e\setminus f$ can be written as normal forms over $X$, namely:
$$
\begin{array}{lcl} e\vee f & = &\vee \{e\in  {\mathcal{E}}\colon \pi(e)\not\in \pi({\mathcal{F}})\} \vee f;\\
e\wedge f &= &\vee\{e\in {\mathcal{E}}\colon \pi(e)\in \pi({\mathcal{F}})\};\\
e\setminus f & = & \vee \{e\in {\mathcal{E}}\colon \pi(e)\not \in\pi({\mathcal{F}})\};\\
e\sqcap f & = & \vee ({\mathcal{E}}\cap {\mathcal{F}}).
\end{array}
$$
\item \label{ii2} Any element of the subalgebra $\langle X\rangle$, generated by $X$, can be written as a normal form over $X$.
\item \label{ii3a} $e\leq f$ if and only if  ${\mathcal{E}}\subseteq {\mathcal{F}}$.
\item \label{ii3b} $e=f$ if and only if ${\mathcal{E}}={\mathcal{F}}$. 
\item \label{ii3} The algebra $\langle X\rangle$ is finite.
\item \label{ii3c}  For any $x\in \langle X\rangle$ the atoms below $x$ are precisely the elements of ${\mathcal{E}}$ where $x= \vee {\mathcal E}$ is the normal form of $x$.
\item \label{ii4} The elementary elements over $X$, which are not equal to zero, are precisely the atoms of $\langle X\rangle$.\end{enumerate}
\end{theorem}

\begin{proof} \eqref{ii1} This is easily established applying the fact that clauses of a normal form commute under $\vee$, distributivity of $\wedge$ and $\sqcap$ over $\vee$ (the former distributivity is a part of the definition of an SBA, and the latter one is part \eqref{sbia5} of Lemma \ref{lem:basic_sbia}), idempotency of the operations $\wedge$, $\sqcap$ and $\vee$, and Proposition \ref{prop:crucial}.

\eqref{ii2} According to Corollary \ref{cor:x}, any $x\in X$ can be written as a normal form. This and part \eqref{ii1} imply the needed claim.

\eqref{ii3a} and \eqref{ii3b}  follow from \eqref{ii1}.

\eqref{ii3} Since $X$ is finite, there are only finitely many elementary elements over $X$ and thus finitely many admissible families of elementary elements.

\eqref{ii3c} follows from  \eqref{ii3a}, and \eqref{ii4}  follows from \eqref{ii2} and~\eqref{ii3a}.

\end{proof}

\subsection{The commutative case} Here we show that in the case where $S$ is commutative, that is, satisfies the identity $x\sqcap y=x\wedge y$,  our theory reduces to the classical theory of elementary conjunctions and disjunctive normal forms in GBAs. 

\begin{lemma} Assume that $S$ satisfies the identity $x\sqcap y=x\wedge y$. Then $e(X,\alpha, A)=0$ for every partition $(X,\alpha)$ with ${\mathrm{rank}}(\alpha)\geq 2$.
\end{lemma}

\begin{proof} Suppose ${\mathrm{rank}}(\alpha)\geq 2$ and let $e(X,\alpha, A)=p\setminus (\vee Q)$, as in \eqref{eq:pq}.
Then $[p]=\wedge [\{a\colon a\in {\mathrm{dom}}({\alpha})\}]$ and for every two distinct blocks $A_i$ and $A_j$ of $\alpha$ we have that $[\wedge (A_i \cup  A_j)]\geq [p]$. In follows that $[\vee Q]\geq [p]$ which yields $p\setminus (\vee Q)=0$.
\end{proof}

Let $(X,\alpha, A)$  a pointed partition where ${\mathrm{rank}}(\alpha)=1$. Then $A$ is the only block of $\alpha$ and thus the set of all such pointed partitions is in a bijection with the set of pairs $(X,A)$ where $A$ is a non-empty subset of $X$. Applying this bijection, we can write $e(X,A)$ for $e(X,\alpha, A)$ (of course, some of the $e(X,A)$ may be equal $0$, if $S$ satisfies some additional identities). The element $e(X,A)$ equals 
$$\sqcap \,\,{\mathrm{dom}}(\alpha)\setminus \vee \,\, (X\setminus {\mathrm{dom}}(\alpha)).$$
Thus non-zero elementary elements are reduced to elementary conjunctions in GBAs. It now follows that the Decomposition Rule in this case reduces to the usual atom decomposition rule in GBAs. For example, if $x\not\in X$ we have
$e(X\cup\{x\}, A)=e(X,A)\vee e(X,A\cup\{x\})$. The normal forms of Subsection  \ref{subs:normal_forms} then reduce to usual full disjunctive normal forms for GBAs. We leave the details to the reader.

\section{Free algebras}\label{s:free}

Let $_{\mathcal L}{\bf SBIA}_{X}$ be the {\em free left-handed SBIA over the generator set } $X$, that is, the algebra of terms in $X$ where two terms are equal if and only if one can be obtained from the another one by a finite number of applications of identities defining the variety of left-handed SBIAs \cite[II.10]{BS}. Let, further, ${\bf GBA}_{X}$ denote the free GBA over $X$.

\subsection{Finite free algebras} Let $X_n=\{x_1,x_2,\dots, x_n\}$.  We denote $_{\mathcal L}{\bf SBIA}_{X_n}$ by $_{\mathcal L}{\bf SBIA}_{n}$ and ${\bf GBA}_{X_n}$ by ${\bf GBA}_n$.

\begin{proposition}\label{prop:free} Let $S$ be a left-handed SBIA and assume that $X_n\subseteq S$. The following statements are equivalent:
\begin{enumerate} 
\item \label{free1}
The subalgebra $\langle X_n\rangle$ is free over $X_n$.
\item \label{free2}All elementary elements over $X_n$ are non-zero.
\item \label{free3} All elementary elements over $X_n$ are non-zero and pairwise distinct.
\end{enumerate}
\end{proposition}

\begin{proof} \eqref{free1} $\Rightarrow$ \eqref{free2} We verify that the equality $e(X_n,\alpha,A)=0$, where $e(X_n,\alpha,A)$ is an arbitrary elementary element, does not follow from the identities defining the variety of left-handed SBIAs. For this, we provide an example of a left-handed SBIA, in which such an equality does not hold. So suppose $e(X_n,\alpha,A)$ is an elementary element. Let $Y={\mathrm{dom}}(\alpha)$ and  $\alpha=\{A_1,\dots, A_k\}$. We show that the equality $e(X_n,\alpha,A)=0$ does not hold in ${\bf{(k+1)}}_L$. For every $i=1,\dots, k$ and every $x\in A_i$ we take the evaluation of $x$ in ${\bf{(k+1)}}_L$ to be equal~$i$. For every $x\in X_n\setminus Y$ we take the evaluation of $x$ to be equal $0$. Then
the evaluation of $e(X_n,\alpha,A)$ equals $m$ where $A=A_m$, so that $e(X_n,\alpha,A)\neq 0$. (Remark that  any evaluation of $e(X_n,\alpha,A)$ in ${\bf m}_L$ with $m\leq k+1$ equals zero. Indeed, in order that the evaluation be non-zero, the elements inside each block must have the same value, and elements from different blocks must have different values.)

 \eqref{free2} $\Rightarrow$ \eqref{free3} follows from part \eqref{iii4} of Corollary \ref{cor:12}. 
 
  \eqref{free3} $\Rightarrow$ \eqref{free1}   Let $x=y$ be an identity holding in $\langle X_n\rangle$. We write every $z\in X_n$ as an orthogonal join of elementary elements, as in Corollary \ref{cor:x}. Since the clauses of such a join commute under $\vee$, applying distributivity of $\wedge$ and $\sqcap$ over $\vee$, idempotency of the operations $\wedge$, $\sqcap$ and $\vee$, and Proposition \ref{prop:crucial}, we rewrite $x$ and $y$ as orthogonal joins of elementary elements. By assumption, all elementary elements are distinct and non-zero.  It follows that any clause appearing in the expression of $x$ appears in the expression of $y$ and wise versa. Thus the equality $x=y$ follows  from the identities defining the variety of left-handed SBIAs.  
  \end{proof}

\begin{example} {\em  We illustrate the proof of the implication \eqref{free1} $\Rightarrow$ \eqref{free2} above. Let $n=5$. Evaluating $x=e(X_5, x_1|x_2x_3|{x_4}, \{x_4\})$ in ${\bf 4}_L$ with
$x_1=1$, $x_2=x_3=2$, $x_4=3$, $x_5=0$ gives us
\begin{multline*}
x= (x_4 \wedge x_1\wedge (x_2\sqcap x_3))\setminus (x_5\vee (x_1\sqcap x_2\sqcap x_3) \vee (x_1\sqcap x_4)\vee (x_2\sqcap x_3\sqcap x_4))=\\
(3\wedge 1\wedge 2)\setminus (0\vee 0\vee 0\vee 0) = 3.
\end{multline*}
At the same time, any evaluation of $x$ in ${\bf 3}_L$ or in ${\bf 2}$ equals $0$.}
\end{example}

\begin{corollary}\label{cor:26} \begin{enumerate}
\item There is a bijective correspondence between atoms of  $_{\mathcal L}{\bf SBIA}_{n}$ and pointed partitions of non-empty subsets of $X_n$. 
\item There is a bijective correspondence between atoms of $_{\mathcal L}{\bf SBIA}_{n}/{\mathcal D}$ and partitions of non-empty subsets of $X_n$.
\end{enumerate}
\end{corollary}

\begin{proof} (1) follows from Proposition \ref{prop:free}, and (2) from the same proposition and Corollary~\ref{cor:d}.
\end{proof}

Since, for ${\bf GBA}_n$, there is a bijective correspondence between its atoms and non-empty subsets of $X_n$, part (2) of Corollary \ref{cor:26}  shows that it is reasonable to call $_{\mathcal L}{\bf SBIA}_{n}/{\mathcal D}$ the  {\em partition analogue} of ${\bf GBA}_n$. It follows that $_{\mathcal L}{\bf SBIA}_{n}$ may be viewed as an `upgrade' of the partition analogue of ${\bf GBA}_n$.

\begin{remark}{\em Proposition \ref{prop:free} tells us that the assignment $(X,\alpha, A)\mapsto e(X,\alpha, A)$ is a bijection if and only the algebra $\langle X \rangle$ is freely generated by $X$. Let $X\subseteq Y \subseteq X_n$ and assume that the algebra $\langle X_n \rangle$ is free.  Then the algebras $\langle X\rangle$ and $\langle Y\rangle$ are free, too. Fix a pointed partition $(X,\alpha,A)$. Then all the elementary elements appearing in equation \eqref{eq:branching_d} of the Decomposition Rule are non-zero and pairwise distinct, so that \eqref{eq:branching_d} models the the passage from $(X,\alpha, A)$ to all pointed partitions $(Y,\beta, A\!\uparrow_{\alpha}^{\beta})$ where $(X,\alpha)\preceq (Y,\beta)$. Similarly, \eqref{eq:branching1} of Corollary \ref{cor:21} models the passage from $(X,\alpha)$ to all $(Y,\beta)$ satisfying $(X,\alpha)\preceq (Y,\beta)$.}
\end{remark}

We are now able to determine the structure and calculate various combinatorial characteristics of $_{\mathcal L}{\bf SBIA}_n$. But first recall that for $n\geq 1$ the $n$th {\em Bell number}, denoted $B_n$,  equals the number of partitions of an $n$-element set. Further,  for $n\geq 1$ and $1\leq k\leq n$,  the {\em Stirling number of the second kind}, $\stirling{n}{k}$, equals the number of partitions of an $n$-element set into $k$ non-empty subsets.

\begin{theorem} \label{th:combinatorial}\mbox{}
\begin{enumerate}
\item \label{c1} $_{\mathcal L}{\mathbf{SBIA}}_n$ has precisely $B_{n+1}-1$ atomic ${\mathcal D}$-classes.
\item \label{c2} $$_{\mathcal L}{\mathbf{SBIA}}_n \simeq {\bf 2}^{\stirling{n+1}{2}} \times {\bf 3}_L^{\stirling{n+1}{3}} \times \dots \times {\bf{(n+1)}}_L^{\stirling{n+1}{n+1}},
$$
Consequently,
$$
|_{\mathcal L}{\mathbf{SBIA}}_n| = 2^{\stirling{n+1}{2}}3^{\stirling{n+1}{3}} \cdots  (n+1)^{\stirling{n+1}{n+1}}.
$$

\item \label{c3}
$_{\mathcal L}{\mathbf{SBIA}}_n$ has  precisely $B_{n+2}-2B_{n+1}$ atoms.

\item \label{c4} There are precisely $2^{\stirling{n+1}{2}}=2^{2^n-1}$ singleton ${\mathcal D}$-classes, so that the center of $_{\mathcal L}{\mathbf{SBIA}}_n$ is isomorphic to ${\bf 2}^{2^n-1}$ which is isomorphic to ${\mathbf{GBA}}_n$.
\end{enumerate}
\end{theorem}

\begin{proof}
(1)  As observed in Corollary \ref{cor:26}, the number of atomic ${\mathcal D}$-classes of $_{\mathcal L}{\mathbf{SBIA}}_n$ equals the number of partitions of non-empty subsets of $X_n$. The latter partitions are in a bijective correspondence with partitions of $X_n\cup \{a\}$, $a\not\in X_n$, of rank at least two:  to a partition $\alpha=\{A_1,\dots, A_k\}$ of $Y\subseteq X_n$, $Y\neq \varnothing$, we assign
the partition $\{A_1,\dots, A_k, (X_n\cup \{a\})\setminus Y\}$ of  $X_n\cup \{a\}$. This assignment is injective and any partition of $X_n\cup \{a\}$ of rank at least two arises this way. Clearly, there are $B_{n+1}-1$ such partitions.

(2)  We apply Proposition \ref{prop:structure}  to $_{\mathcal L}{\mathbf{SBIA}}_n$ and the fact that a finite primitive left-handed SBA $D^0$ is isomorphic to ${\bf{(k+1)}}_L$ where $k=|D|$. Thus, for each $k\geq 1$,  the factorization of $_{\mathcal L}{\mathbf{SBIA}}_n$ into a product of primitive algebras contains as many direct factors ${\bf{(k+1)}}_L$ as there are atomic ${\mathcal D}$-classes of cardinality $k$. By part \eqref{iii3} of Corollary \ref{cor:12}, the latter number equals the number of partitions of non-empty subsets of $X_n$ of rank $k$. Applying the same assignment as in the proof of part (1) above, we see that this equals the number of partitions of $X_n\cup \{a\}$, where $a\not\in X_n$, of rank $k+1$. Since latter number equals $\stirling{n+1}{k+1}$, \eqref{c2} is proved.

(3) Since ${\bf{(k+1)}}_L$ has precisely $k$ atoms, part (2) above implies that the number of atoms in  $_{\mathcal L}{\mathbf{SBIA}}_n$ equals
\begin{equation}\label{eq:atoms}
\sum_{k=2}^{n+1}(k-1)\stirling{n+1}{k}.
\end{equation}
We show that this number equals $B_{n+2}-2B_{n+1}$.
Recall \cite{W} that a {\em Bell Polynomial} $B_n(x)$ is defined by
$B_n(x)=\sum_{k=0}^n \stirling{n}{k}x^k$ for $x\in {\mathbb R}$. In particular, $B_n(1)=B_n$. We rewrite \eqref{eq:atoms}:
$$
\sum_{k=2}^{n+1}(k-1)\stirling{n+1}{k}=\sum_{k=2}^{n+1}k\stirling{n+1}{k} - \sum_{k=2}^{n+1}\stirling{n+1}{k}.
$$
The first of the two sums in the right-hand side above equals $B_{n+1}'(1)-1$ (where $B_{n+1}'(1)$ is the derivative of $B_{n+1}(x)$ evaluated at $x=1$), and the second sum equals $B_{n+1}-1$. It follows that the  sum in \eqref{eq:atoms} equals $B_{n+1}'(1) - B_{n+1}$. But from the recurrence relation for Bell polynomials \cite{W} we have
$B_{n+1}'(1) + B_{n+1}=B_{n+2}$. Thus the needed sum equals $B_{n+2}-2B_{n+1}$, as required.

(4) By part (2) of Corollary \ref{cor:12}, a ${\mathcal D}$-class $[(X_n,\alpha, A)]$ is singleton if and only $\alpha$ is of rank one (so that $A$ is the only block of $\alpha$). The number of such partitions equals the number of non-empty subsets of $X_n$, $2^n-1$. (and (2) above gives the known equality $2^n-1={\stirling{n+1}{2}}$.)  By Proposition \ref{prop:10}, the center of the product $D_1^0\times \dots \times D_m^0$ equals  the products of those factors of $D_1^0, \dots, D_m^0$ which are isomorphic to ${\bf 2}$. Thus the center of $_{\mathcal L}{\mathbf{SBIA}}_n$ is isomorphic to ${\bf 2}^{2^n-1}$ which is isomorphic to ${\mathbf{GBA}}_n$.
 \end{proof}
 
 Part (1) of the following result is not surprising, given that ${\mathcal D}$ is not a congruence on $_{\mathcal L}{\mathbf{SBIA}}_n$ but only on its SBA reduct. Let $\gamma \colon  {_{\mathcal L}{\mathbf{SBIA}}_n}\to \,{_{\mathcal L}{\mathbf{SBIA}}_n}/\theta$ be the universal morphism to the maximal commutative quotient.  As $_{\mathcal L}{\mathbf{SBIA}}_n$ is free over $X_n$, ${_{\mathcal L}{\mathbf{SBIA}}_n}/\theta$ is a free  GBA over $\{\gamma(x)\colon x\in X_n\}$. Part (2) of the result below shows that $\gamma(x)$, $x\in X_n$, are all pairwise distinct.

 \begin{corollary} \label{cor:free11} \mbox{}
 \begin{enumerate}
 \item $_{\mathcal L}{\mathbf{SBIA}}_n/{\mathcal D}$, the partition analogue of  ${\mathbf{GBA}}_n$, is isomorphic to ${\bf 2}^{B_{n+1}-1}$ and  thus is not isomorphic to a free GBA.
 \item The maximal commutative quotient of $_{\mathcal L}{\mathbf{SBIA}}_n$ is isomorphic to ${\bf 2}^{2^n-1}$ which is isomorphic to ${\mathbf{GBA}}_n$. 
 \end{enumerate}
 \end{corollary}
 
 \begin{proof} (1) follows from part \eqref{c1} of Theorem \ref{th:combinatorial}.
 (2) is a consequence of part (3) of Proposition \ref{prop:10} and part \eqref{c4} of Theorem \ref{th:combinatorial}.
 \end{proof}

 \begin{proposition}\label{prop:gen}
 Let $S$ be a left-handed SBIA and $n\geq 1$. Then $S$ can be generated by $n$ its elements (as an SBIA) if and only if $S$ is a isomorphic to a direct factor of $_{\mathcal L}{\mathbf{SBIA}}_n$. In other words, if $S\simeq {\bf 2}^{k_2}\times {\bf 3}_L^{k_3}\dots \times {\bf m}_L^{k_m}$
 with $k_2,k_3,\dots,k_m\geq 0$, then $S$ can be generated by $n$ elements of $S$ if and only if $m\leq n+1$ and
 $k_i\leq \stirling{n+1}{i}$ for each $i=1,\dots, m$.
 \end{proposition}

 \begin{proof} $S$ can be generated by $n$ elements if and only if there is a surjective homomorphism, $\psi$, from $_{\mathcal L}{\mathbf{SBIA}}_n$ onto $S$. From the duality theorem for left-handed SBIAs \cite{Kud, BCV} applied for the finite case where topologies are discrete, it follows that there is an injective map from atoms of $S/{\mathcal D}$ to atoms of $_{\mathcal L}{\mathbf{SBIA}}_n/{\mathcal D}$. (This is the map $\overline{\psi}^{-1}$, where $\overline{\psi}\colon _{\mathcal L}{\mathbf{SBIA}}_n/{\mathcal D}\to S/{\mathcal D}$ is the surjective GBA homomorphism induced by $\psi$.) This gives rise to an injective map, $\varphi$, from  atomic ${\mathcal D}$-classes of $S$ to atomic ${\mathcal D}$-classes of $_{\mathcal L}{\mathbf{SBIA}}_n$.
Moreover, for every atomic ${\mathcal D}$-class class $D$ of $S$ there is a bijection $\varphi(D)\to D$.  Indeed, by the duality theory \cite{Kud}, the corresponding \'etale space cohomomorphism is injective, so that each map $\varphi(D)\to D$ is injective. Each such a map must be also surjective, as $\psi$ is surjective. The result follows.
 \end{proof}
 
 The construction of a surjective homomorphism from $_{\mathcal L}{\mathbf{SBIA}}_n$ to any given $S$, whenever it exists, can be tracked back from our theory. An example of such a construction is given below.
 
 \begin{example}
 {\em Let $S={\bf 3}_L^6$. Since $\stirling{4}{3}=6$, $S$ is isomorphic to a direct factor of $_{\mathcal L}{\mathbf{SBIA}}_3$. The two-element atomic ${\mathcal D}$-classes of $_{\mathcal L}{\mathbf{SBIA}}_3$ are of the  form $\{(X_3,\alpha, A), (X_3,\alpha, B)\}$ where $\alpha=\{A,B\}$. Note that ${\bf 3}_L^6$, considered as a direct factor of $_{\mathcal L}{\mathbf{SBIA}}_3$, is generated by respective restrictions, $y_1,y_2$ and $y_3$, of $x_1,x_2$ and $x_3$. Applying Corollary \ref{cor:x}, these restrictions can be written, for each $i=1,2,3$, as
 $$y_i=\vee \{(X_3,\alpha, A)\colon \mathrm{rank}(\alpha)=2 \text{ and } y_i\in A\}.$$ To write down an explicit isomorphism between the direct factor ${\bf 3}_L^6$ of $_{\mathcal L}{\mathbf{SBIA}}_3$ and $S$, we make the following assignments:
$$\begin{array}{llllll}
(X_3, x_1|x_2, \{x_1\}) & \mapsto & (1,0,0,0,0,0),\, & (X_3, x_1|x_2, \{x_2\}) & \mapsto & (2,0,0,0,0,0),\\
(X_3, x_1|x_3, \{x_1\}) & \mapsto & (0,1,0,0,0,0),\, & (X_3, x_1|x_3, \{x_3\}) & \mapsto & (0,2,0,0,0,0),\\
(X_3, x_2|x_3, \{x_2\}) & \mapsto & (0,0,1,0,0,0),\, & (X_3, x_2|x_3, \{x_3\})& \mapsto & (0,0,2,0,0,0),\\
(X_3, x_1x_2|x_3, \{x_1,x_2\})& \mapsto & (0,0,0,1,0,0),\, & (X_3, x_1x_2|x_3, \{x_3\})& \mapsto & (0,0,0,2,0,0),\\
(X_3, x_1x_3|x_2, \{x_1,x_3\})& \mapsto & (0,0,0,0,1,0),\, & (X_3, x_1x_3|x_2, \{x_2\})& \mapsto & (0,0,0,0,2,0),\\
(X_3, x_1|x_2x_3, \{x_1\})& \mapsto & (0,0,0,0,0,1),\, & (X_3, x_1|x_2x_3, \{x_2,x_3\})& \mapsto & (0,0,0,0,0,2).
 \end{array}
 $$
 Then $y_1\mapsto (1,1,0,1,1,1)$, $y_2\mapsto (2,0,1,1,2,2)$, $y_3\mapsto (0,2,2,2,1,2)$, and the obtained three elements generate $S={\bf 3}_L^6$ as an SBIA. Note that $S$ can not be generated by three generators as an SBA, but at least by four generators, see \cite{KL}.}
 \end{example}

 \subsection{Infinine free algebras} Let now $X$ be an infinite set. If $a$ is a term over a finite subset $Y$ of $X$, and $b$ is a term over a finite subset $Z$ of $X$, Theorem \ref{th:normal_forms} allows us to rewrite them as normal forms over $Y$ and $Z$, respectively, and then, using the Decomposition Rule, rewrite each as a normal form over the finite set $Y\cup Z$.  Then the results of applications of operations $\wedge$, $\vee$, $\setminus$ and $\sqcap$ to $a$ and $b$ can be calculated using Theorem \ref{th:normal_forms} and have natural interpretations in terms of pointed partitions and their containments.

 The following is a standard consequence of Proposition \ref{prop:free} and the fact that a term over any alphabet $X$ involves only a finite number of variables.

\begin{proposition}\label{prop:free_infinite} Let $S$ be a left-handed SBIA and assume that $X\subseteq S$. The following are equivalent:
\begin{enumerate}
\item The subalgebra $\langle X\rangle$ is free over $X$.
\item  All elementary elements over all finite subsets of $X$ are non-zero.
\item For any finite subset $Y\subseteq X$ the subalgebra $\langle Y\rangle$ is free over $Y$.
\end{enumerate}
\end{proposition}

The following result is proved similarly as the corresponding result of \cite[Section 6]{KL}.

\begin{proposition}\label{prop:center_atoms} Let $X$ be an infinite set.
\begin{enumerate}
\item $_{\mathcal L}{\bf SBIA}_{X}$ is atomless.
\item The center of $_{\mathcal L}{\bf SBIA}_{X}$ is trivial.
\end{enumerate}
\end{proposition}

We now present a construction of $_{\mathcal L}{\bf SBIA}_{X}$. Let
$${\mathcal X} = \{(X,\alpha)\colon \alpha\in {\mathcal P}(Y) \text{ where } Y\subseteq X \text{ and } Y \neq\varnothing\}$$
be the set of all partitions of all non-empty subsets of $X$ and
$$
\Omega = \{(X,\alpha, A)\colon (X,\alpha) \in {\mathcal X} \text{ and } A \in \alpha\}.
$$
be the set of pointed such partitions.
We define $p\colon \Omega\to {\mathcal X}$ by $p(X,\alpha, A)=(X,\alpha)$. Let ${\bf S}_{\Omega}$ be the class of subsets $U$ of $\Omega$ for which the restriction of the map $p$ to $U$ is injective. On ${\bf S}_{\Omega}$ we define the binary operations  $\vee$, $\wedge$, $\setminus$ and $\sqcap$ by:
{
$$
\begin{array}{lcl}
U \wedge V & = & \{(X,\alpha, A) \in U\colon (X,\alpha) \in  p(U) \cap p(V)\},\\
U\vee V & = & (U\setminus V)\cup V = \{(X,\alpha,A)\in U\cup V\colon (X,\alpha) \in p(U)\setminus p(V) \text{ or } (X,\alpha) \in p(V)\},\\
U \setminus  V & = &\{(X,\alpha, A) \in U\colon (X,\alpha) \in p(U)\setminus  p(V)\},\\
U\sqcap V & =  & U\cap V.
\end{array}
$$
}
It is easy to verify that
$({\bf S}_{\Omega};  \wedge,\vee \setminus, \varnothing, \sqcap)$ is a left-handed SBIA.
We next define $i\colon X \to {\bf S}_{\Omega}$ by
$$i(x) = \{(X,\alpha, A)\colon x\in {\mathrm{dom}}(\alpha) \text{ and } x\in A\}.$$

 This map is clearly injective. We next let
 $\overline{X} = \{i(x) \colon x \in X\}$ and let ${\bf S}_X = \langle \overline{X}\rangle$ be the subalgebra of ${\bf S}_{\Omega}$ generated by $\overline{X}$.

\begin{theorem}\label{th:5.8}
${\bf S}_X$ is freely generated by $\overline{X}$.
\end{theorem}
\begin{proof} Given a finite subset $X_n=\{x_1, x_2, \dots, x_n\}$ of $X$, we show that each elementary element over  $X_n$, when evaluated on $\{i(x_1), i(x_2), \dots, i(x_n)\}$, is non-empty. We observe that:
{$$
\begin{array}{lcl}
i(x)\wedge i(y) & = & \{(X,\alpha, A) \in \Omega\colon x,y\in {\mathrm{dom}}(\alpha) \text{ and } x\in A\},\\
i(x) \setminus i(y)& = & \{(X,\alpha, A) \in \Omega\colon x\in {\mathrm{dom}}(\alpha) \text{ and } y\not\in {\mathrm{dom}}(\alpha) \},\\
i(x)\sqcap i(y) & = & \{(X,\alpha, A) \in \Omega\colon x,y\in {\mathrm{dom}}(\alpha) \text{ and } x,y\in A\}.

\end{array}
$$
}
As a consequence we obtain that the evaluation of $e(X_n,\alpha,A)$ on $\{i(x_1), i(x_2), \dots, i(x_n)\}$ equals
\begin{equation*}
\{(X,\beta, A\!\uparrow_{\alpha}^{\beta}) \in \Omega \colon  (X_n,\alpha)\preceq (X,\beta)\}
\end{equation*}
which is non-empty.
By Proposition \ref{prop:free_infinite}, $\langle \overline{X}\rangle$ is isomorphic to $_{\mathcal L}{\bf SBIA}_{X}$. \end{proof}

\begin{corollary}\label{cor:over_d}
Any element of $_{\mathcal L}{\bf SBIA}_{X}/{\mathcal D}$ is a join of a finite number of 
elements of the form $\mathrm{i}(Y,\alpha)$ with $Y$ being a finite non-empty subset of $X$ and $\alpha$ a partition of a non-empty subset of $Y$ where
$$
\mathrm{i}(Y,\alpha)=\{(X,\beta)\in {\mathcal{X}}\colon (Y,\alpha)\preceq (X,\beta)\}.
$$
\end{corollary}

\subsection{Countable generating set} Let $X$ be a countable set. It was shown in \cite{KL} that $_{\mathcal L}{\bf SBA}_{X}/{\mathcal D}$, the free left-handed SBA over $X$, is isomorphic to ${\bf GBA}_{X}$. We now show, that in contrast to Corollary \ref{cor:free11},  also $_{\mathcal L}{\bf SBIA}_{X}/{\mathcal D}\simeq {\bf GBA}_{X}$.

A {\em Cantor set} is a totally disconnected metrizable compact space without isolated points. It is well-known that any two such spaces are homeomorphic and hence any of these spaces can be called `the Cantor set'.  A classification of all ultrametrics on a Cantor set was given by Michon \cite{Michon} (see also \cite{PB}). This result implies that the boundary of any Cantorian tree is homeomorphic to the Cantor set where a tree is {\em Cantorian} if it is rooted, locally finite (that is, each vertex has a finite number of children), has no dangling vertices (that is, vertices without children) and each vertex has a descendant with more than one child. The basis of the topology on the boundary of a Cantorian tree is formed by the sets $[v]$, where $v$ runs through the vertices of the tree, the set $[v]$ consisting of all points of the boundary which pass through $v$ (recall that points of the boundary are infinite paths $v_0v_1v_2......$ where $v_i$ belongs to the $i$th level of the tree for each $i$).

The Decomposition Rule for $_{\mathcal L}{\bf SBIA}_{X}/{\mathcal D}$, for $X={\mathbb N}$, leads to the construction of the {\em infinite partition tree} which can be looked at as the `partition analogue' of the Cantor tree.  Its level $0$ is just the empty partition of the empty set, and for each $i\geq 1$ vertices of level $i$ are partitions of subsets of $[i]=\{1,2,\dots, i\}$. A vertex $([i],\alpha)$ of level $i$ is connected with a vertex $([i+1],\beta)$ of level $i+1$ if and only if $([i],\alpha)\preceq ([i+1],\beta)$. The first four levels of the infinite partition tree are shown on Figure 1. Elements of the boundary of this tree may be identified with partitions of subsets of $\mathbb{N}$. The set $[([i],\alpha)]$, where $([i],\alpha)$ is a vertex, consists of all partitions $({\mathbb N},\beta)$ such that $([i],\alpha)\preceq ({\mathbb N},\beta)$. 

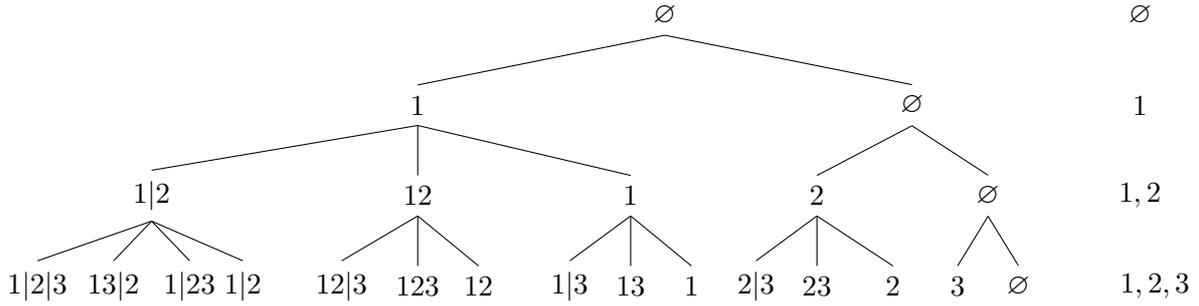
\begin{figure}
\begin{center}
\begin{tikzpicture}[level distance=1.2cm]

    \node (Root)  {$\varnothing$}
        child  [sibling distance=65 mm] {
        node {$1$}
        child [sibling distance=35 mm]{ node {$1|2$}
           child [sibling distance=10 mm] { node {$1|2|3$}}
           child [sibling distance=10 mm] { node {$13|2$}}
           child [sibling distance=10 mm] { node {$1|23$}}
           child [sibling distance=8 mm]{node {$1|2$}}
        }
        child [sibling distance=25 mm] { node {$12$}
           child  [sibling distance=10 mm] {node {$12|3$}}
           child [sibling distance=10 mm] {node {$123$}}
           child [sibling distance=8 mm] {node {$12$}}
        }
        child  [sibling distance=28 mm] { node {$1$}
            child  [sibling distance=8 mm] {node {$1|3$}}
           child [sibling distance=8 mm] {node {$13$}}
           child [sibling distance=8 mm] {node {$1$}}}
    }
    child  [sibling distance=65 mm] {
        node {$\varnothing$}
        child  [sibling distance=25 mm]{ node {$2$}
           child  [sibling distance=8 mm] {node {$2|3$}}
           child [sibling distance=8 mm] {node {$23$}}
           child [sibling distance=10 mm] {node {$2$}}
        }
        child  [sibling distance=20 mm] { node {$\varnothing$}
        child  [sibling distance=8 mm]  {node{$3$}}
        child  [sibling distance=8 mm]  {node {$\varnothing$}}
        }
    };
   
   \begin{scope}
     \path (Root    -| Root-2-2) ++(20mm,0) node {$\varnothing$};
     \path (Root-1  -| Root-2-2) ++(20mm,0) node {$1$};
     \path (Root-1-1-| Root-2-2) ++(20mm,0)  node {$1,2$};
     \path (Root-1-1-1-| Root-2-2-2) ++(18mm,0)  node {$1,2,3$};
   \end{scope}

\end{tikzpicture}
\end{center}
\caption{The first four levels of the infinite partition tree}
\end{figure}

A {\em locally compact Cantor set} is a Cantor set  with one point removed, with respect to the subspace topology. It is well known that the dual (under the classical Stone duality) Boolean algebra of the Cantor set is isomorphic to the free Boolean algebra on countably many generators, and the dual  generalized Boolean algebra of the locally compact Cantor set is isomorphic to the free generalized Boolean algebra on countably many generators. Thus, removing the rightmost path from the boundary of the infinite partition tree (which corresponds to the empty partition of the empty subset of ${\mathbb N}$) we obtain precisely the set ${\mathcal X}$ for $X={\mathbb N}$ and, moreover, the sets $[([i],\alpha)]$ are precisely the sets $\mathrm{i}([i],\alpha)$ from Corollary~\ref{cor:over_d}. Since, applying the Decomposition Rule, any element $\mathrm{i}(Y,\alpha)$ from Corollary~\ref{cor:over_d} can be written as an orthogonal join of elements $\mathrm{i}([i], \gamma)$ (with $i$ being the maximum element of $Y$), it follows that $_{\mathcal L}{\bf SBIA}_{X}/{\mathcal D}\simeq {\bf GBA}_{X}$.

\end{document}